\documentclass[12pt]{article}
\usepackage{amsmath,amsthm,amsfonts,amscd,amsxtra,amsopn,amssymb,verbatim,pdfsync,hyperref,array}
\usepackage[headings]{fullpage}
\usepackage{enumitem}[shortlabels]
\usepackage{mathabx}
\usepackage[mathscr]{eucal}
\usepackage{graphicx}
\usepackage{epstopdf}
\usepackage{tikz}
\usetikzlibrary{arrows,decorations.pathmorphing,backgrounds,fit,positioning,shapes.symbols,chains}
\usetikzlibrary{shapes,arrows}
\usetikzlibrary{calc}
\usetikzlibrary{cd}

\newtheorem{prop}{Proposition}
\numberwithin{prop}{section} 

\newtheorem{thm}[prop]{Theorem}
\newtheorem{cor}[prop]{Corollary}

\newtheorem{question}[prop]{Question}

\newtheorem{ddefn}[prop]{Definition}
\newtheorem{eex}[prop]{Example}
\newtheorem{eexs}[prop]{Examples}
\newtheorem{rrem}[prop]{Remark}
\newtheorem{eexercise}[prop]{Exercise}
\newtheorem{con}[prop]{Conjecture}
\newtheorem{hhome}[prop]{Homework}
\newtheorem{nnumber}[prop]{}

\newenvironment{defn}{\begin{ddefn}\rm}{\end{ddefn}}
\newenvironment{ex}{\begin{eex}\rm}{\end{eex}}

\newenvironment{conj}{\begin{con}\rm}{\end{con}}

{\catcode`@=11\global\let\c@equation=\c@prop}

\newcommand\C{\mathbb C}
\newcommand\D{\mathbb D}

\renewcommand\H{\mathbb H}

\renewcommand\P{\mathbb P}
\newcommand\Q{\mathbb Q}
\newcommand\R{\mathbb R}

\newcommand\Z{\mathbb Z}

\newcommand\cM{\mathcal M}

\newcommand\Iff{if and only if }
\newcommand{\ra}{\longrightarrow}

\newcommand{\into}{\hookrightarrow}

\newcommand\coker{\operatorname{coker}}
\newcommand\sdim{\operatorname{sdim}}
\newcommand\Cl{C\ell} 
\newcommand\fM{\mathfrak M}
\newcommand\p{\partial}
\newcommand\MSpin{\operatorname{MSpin}}
\newcommand\im{\operatorname{im}}
\newcommand{\id}{\operatorname{id}}

\def\wt{\widetilde}
\def\wh{\widehat}

\newcommand\RP{{\R\P}}
\newcommand\CP{{\C\P}}
\newcommand\HP{{\H\P}}
\newcommand\SU{{\operatorname{SU}}}
\newcommand\SO{{\operatorname{SO}}}
\newcommand\Spin{{\operatorname{Spin}}}
\newcommand\spin{{\operatorname{Spin}}}
\newcommand\PSp{{\operatorname{PSp}}}
\newcommand\Sp{{\operatorname{Sp}}}
\newcommand\MSO{{\operatorname{MSO}}}

\newcommand\BSpin{{\operatorname{BSpin}}}
\newcommand\HZ{{\rm{H}\Z}}
\newcommand\MG{{\operatorname{MG}}}

\newtheorem{cconstruction}[prop]{Construction}
\newenvironment{construction}{\begin{cconstruction}\rm}{\end{cconstruction}}

\newcommand\point{\operatorname{point}}
\newcommand\ko{{\operatorname{ko}}}
\newcommand\KO{{\operatorname{KO}}}
\newcommand\dist{{\operatorname{dist}}}
\newcommand\sP{{\mathscr P}}
\newcommand\vol{{\operatorname{vol}}}
\newcommand\psc{positive scalar curvature}
\newcommand\pscm{positive scalar curvature metric}
\def\nb-{\nobreakdash-}
\newcommand\bra[1]{\langle #1\rangle}

\title{Positive scalar curvature -- constructions and obstructions}
\author{Stephan Stolz}
\begin{document}
\maketitle

\abstract 
This is a survey of the current state of the question ``Which closed connected manifolds of dimension $n\ge 5$ admit Riemannian metrics whose scalar curvature function is everywhere positive?" The introduction gives a brief overview of these results, while the body of the paper discusses the methods used in the proofs of these results. We mention the two flavors of  topological obstructions to the existence of \pscm s: one is a consequence of the Weizenb\"ock formula for the Dirac operator, the other is obtained by considering stable minimal hypersurfaces. We talk about geometric constructions of \pscm s (the surgery/bordism theorem), which shows that the answer to the question above depends only the bordism class of the manifold in a suitable bordism group. 
Via the Pontryagin-Thom construction this can be translated into stable homotopy theory, and solved completely in some cases, in particular for simply connected manifolds, or manifolds with very special fundamental groups. The last section discusses some open questions. 

\tableofcontents

\section{Introduction}

Among the various flavors of curvature of a Riemannian manifold $M$, the scalar curvature is the most basic: it is a smooth function $s\colon M\to\R$ whose value at a point $x\in M$ controls the volume of small balls with center $x$ in the sense that $s(x)$ shows up in the first interesting coefficient of the Taylor  series expansion
\[
\frac{\vol\  B_r(x,M)}{\vol\  B_r(0,\R^n)}
=1-\frac{s(x)}{6(n+2)}r^2+\dots
\]
Here  $\vol\  B_r(x,M)$ is the volume of the geodesic ball of radius $r$ around $x$ in $M$, and $\vol\  B_r(0,\R^n)$ is the volume of the ball of the same radius in $\R^n$ for  $n=\dim M$.

\begin{nnumber} \label{quest:main}
{\bf Question.} Let $M$ be a closed connected manifold. Does $M$ carry a \pscm, i.e., a Riemannian metric such that $s(x)>0$ for all $x\in M$?
\end{nnumber}

This paper does not contain new results concerning this question, but rather its goal is to give a brief overview of known results addressing this question in the introduction and to  outline the methods used in their proof in the body of the paper. In the last section we mention some open problems. We have attempted to give precise references for readers curious about details. 

We will restrict our discussion to manifolds of dimension $n\ge 5$. The reason is that the flavor of the answers and the tools used very much depend on $n$, like in the classification of manifolds up to diffeomorphism. The crucial tool for the diffeomorphism classification 
of manifolds of dimension $n\ge 5$ is the $s$\nb-cobordism theorem, according to which two such manifolds which are $s$\nb-cobordant are in fact diffeomorphic. The analogous tool for studying Question \ref{quest:main} is Theorem \ref{thm:bordism} according to which a cobordism between manifolds $M$ and $N$  of dimension $\ge 5$, equipped with suitable additional structures, implies that if $N$ carries a  \pscm, then so does $M$. Also, there are diffeomorphism invariants, like the Seiberg-Witten invariant, which are only defined  for manifolds of dimension $n=4$. The Seiberg-Witten invariant is also relevant for understanding \pscm s on $4$\nb-manifolds, since it vanishes for any closed $4$\nb-manifold that carries  a \pscm.

The reader might wonder about the implicit bias in question \ref{quest:main} by asking about {\em positive} scalar curvature metrics. The reason is that if $s\colon M\to \R$ is a smooth function on any closed manifold of dimension $n\ge 3$ with $s(x)<0$ at some point $x\in M$, then there exists  a Riemannian metric on $M$ with scalar curvature function $s$ \cite[Thm.\ 6.4(b)]{KW1975}. Moreover, if a closed manifold $M$ of dimension $n\ge 3$ carries a \pscm, then {\em any} smooth function $s\colon M\to \R$ is realized as the scalar curvature function of some metric \cite[Thm.\ 6.4(a)]{KW1975}. So the situation in dimension $n\ge 3$ is radically different than in dimension $2$, where having everywhere positive {\em and} everywhere negative functions as the scalar curvature on the same $2$\nb-manifold is ruled out by the Gauss-Bonnet Theorem.

\medskip

Lichnerowicz observed that the Dirac operator on a closed Riemannian spin manifold $M$ with \psc\ is invertible, and hence its index vanishes. By the Atiyah-Singer Index Theorem, the index of the Dirac operator is equal to a topological invariant $\wh A(M)$ (known as Hirzebruch's $\wh A$\nb-genus), and hence $\wh A(M)$ vanishes for if $M$ carries a \pscm\ \cite{Li1963}. Including a later refinement by Hitchin \cite{Hi1974}leads to the following result. 

\begin{thm}{\bf (Lichnerowicz, Hitchin)}\label{thm:LH}
Let $M$ be a closed Riemannian spin manifold of dimension $n$ with \pscm, then its Atiyah invariant $\alpha(M)\in \KO_n$ vanishes. 
\end{thm}

Here $\KO_*(\ )$ is the generalized homology theory known as {\em real $K$\nb-homology}, and $\KO_n=\KO_n(\point)$ is the $n$\nb-th real $K$\nb-homology group of a point. The invariant $\alpha(M)$ has a topological as well as an index theory definition.
 Topologically, the spin structure on $M$ determines a fundamental class $[M]^{\spin}\in \KO_n(M)$ and $\alpha(M)=p_*[M]^{\spin}$, where $p_*\colon \KO_n(M)\to \KO_n(\point)$ is the homomorphism induced by the projection map $p\colon M\to \point$. From an index theory point of view, the invariant $\alpha(M)$ can be viewed as a ``refined index'' (with values in $\KO_n$ rather than $\Z$) of the {\em Clifford linear Dirac operator} on $M$, as explained in section \ref{ssec:Cliff_Dirac}, see equation \eqref{eq:alpha}.

For $n\equiv 0\mod 4$, the invariant $\alpha(M)\in \KO_n\cong \Z$ is equal to  $c\wh A(M)$ where $c=1$ for  $n\equiv 0\mod 8$ and $c=\frac 12$ for $n\equiv 4\mod 8$. It is interesting to note that for $n\equiv 1,2\mod 8$ and $n\ge 9$ there is an $n$\nb-manifold $\Sigma$ with $0\ne \alpha(\Sigma)\in \KO_n\cong \Z/2$ which is homeomorphic, but not diffeomorphic to the sphere $S^n$. In particular, $S^n$ carries a \pscm\ while the homeomorphic manifold $\Sigma$ does not! This shows that Question \ref{quest:main} is subtle: the answer in general depends on the smooth structure of $M$. 

There is a generalization of the Lichnerowicz-Hitchin Theorem \ref{thm:LH} due to Rosenberg based on an index invariant $\alpha(M,f)\in \KO_n(C^*\pi)$ associated to a closed spin $n$\nb-manifold $M$ equipped with a map $f\colon M\to B\pi$ to the classifying space $B\pi$ of a discrete group $\pi$ \cite[Proof of Thm.\ 3.4]{Ro1986III}. Here $\KO_n(C^*\pi)$  is the real $K$\nb-theory of the reduced group $C^*$\nb-algebra $C^*\pi$, which is a completion of the real group ring $\R\pi$. If $\pi$ is the trivial group, this index invariant agrees with $\alpha(M)$.

\begin{thm}{\bf (Rosenberg)}\label{thm:Ro}
If $M$ carries a \pscm, then $\alpha(M,f)$ vanishes.
\end{thm}

\medskip

After discussing {\bf obstructions} to \pscm s, we now turn to {\bf constructions} of manifolds which carry such metrics.  A major breakthrough was obtained independently by Gromov-Lawson and Schoen-Yau who proved that if a closed manifold $M$ is obtained by a surgery of codimension $\ge 3$ from a manifold $N$ which carries a \pscm, than also $M$ carries  a \pscm\ \cite[Thm.\ A]{GL1980}, \cite[Cor.\ 6]{SY1979a}. Gromov and Lawson made the crucial observation that this implies that the answer to Question \ref{quest:main} for a simply connected manifold $M$ of dimension $n\ge 5$ depends only on its bordism class in a suitable bordism group \cite[Thm.\ B, Thm.\ C]{GL1980}. This was later extended by Rosenberg to certain manifolds with non-trivial fundamental group \cite[Thm.\ 2.2, Thm.\ 2.13]{Ro1986II}. 

Let $\Omega_n^\SO$ resp.\ $\Omega_n^\spin$ be the bordism group of $n$\nb-dimensional closed manifolds equipped with an orientation resp.\ spin structure. More generally, for a topological space $X$, and $G=\SO$ or $G=\spin$, let $\Omega_n^G(X)$ be the bordism classes of pairs $(M,f)$ where $M$ is an oriented resp.\ spin $n$\nb-manifold and $f\colon M\to X$ is a map. Let $\Omega_n^{G,+}(X)$ be the subgroup of $\Omega_n^G(X)$ consisting of those bordism classes that can be represented by pairs $(M,f)$ such that $M$ carries a \pscm.

\begin{thm}{\bf (Gromov and Lawson, Rosenberg).}\label{thm:bordism}
 Let $M$ be a closed connected $n$\nb-manifold, $n\ge 5$, with fundamental group $\pi$, and let  $u\colon M\to B\pi$  be the classifying map of the universal covering $\wt M\to M$. 
\begin{enumerate}[label=\normalfont(\roman*),topsep=2pt,itemsep=-2pt]
\item If $M$ is spin, then it carries a \pscm\ \Iff  $[M,u]\in \Omega_n^{\spin,+}(B\pi)$.
\item If $M$ is oriented and $\wt M$ is non-spin, then $M$ admits a \pscm\ \Iff $[M,u]\in \Omega_n^{\SO,+}(B\pi)$. 
\end{enumerate}
\end{thm}

More generally, for {\em any} closed manifold $M$ with fundamental group $\pi$, the pair $(M,u)$ represents an element in a {\em twisted} version of the bordism group $\Omega_n^\spin(B\pi)$ (if $\wt M$ is spin) resp.\ $\Omega_n^\SO(B\pi)$ (if $\wt M$ is non-spin). There are generalizations of the theorem above \cite[Thm.\ 3.3 and Examples 3.6 \& 3.7]{RS1994}.

The oriented bordism ring $\Omega_*^\SO$ was determined by Wall \cite{Wa1960}, who also provided a list of explicit oriented manifolds whose bordism classes are  multiplicative generators for $\Omega_*^\SO$. Gromov and Lawson showed that all of these manifolds admit \pscm s, which by the theorem above implies \cite[Cor.\ C]{GL1980}:

\begin{cor}{\bf (Gromov and Lawson).}\label{cor:GL} Every closed simply connected non-spin manifold of dimension $n\ge 5$ admits a \pscm.
\end{cor} 

For a closed spin $n$\nb-manifold $M$ the vanishing of the index obstruction $\alpha(M)\in \KO_n$ is a necessary condition for the existence of a \pscm\ on $M$ by the Lichnerowicz-Hitchin theorem \ref{thm:LH}.
Gromov and Lawson conjectured  that it is also a sufficient condition for simply connected manifolds of dimension $\ge 5$ \cite[p.\ 424]{GL1980}. This  was later proved by the author \cite[Thm.\ A]{St1992}:

\begin{thm}{\bf (Stolz).}\label{thm:stolz} Let $M$ be a simply connected closed spin manifold of dimension $n\ge 5$. Then $M$ admits a \pscm\ \Iff\ $\alpha(M)=0$.
\end{thm}

The spin bordism groups were calculated by Anderson, Brown and Peterson \cite{ABP1967}.  However, unlike for the oriented bordism ring, there is no list of explicit spin manifolds which multiplicatively generate $\Omega_*^\spin$ (or the ideal consisting of the bordism classes of spin manifolds with vanishing $\alpha$\nb-invariant). In fact, to the author's knowledge, the multiplicative structure of $\Omega_*^\spin$ has not been completely determined. 

The theorem above is a corollary of a statement about spin bordism, according to which every closed spin manifold $M$ with $\alpha(M)=0$ is spin bordant to the total space of a fiber bundle, whose  fiber is the quaternionic projective plane $\HP^2$, and whose structure group is the isometry group of the standard metric on $\HP^2$ (see Theorem \ref{thm:transfer}). This implies Theorem \ref{thm:stolz}, since total spaces of such fiber bundles carry \pscm s by Observation \ref{ob:fib}(2).

The bordism statement is proved using the Pontryagin-Thom construction to express bordism groups as homotopy groups of suitable Thom spectra, and using the Adams spectral sequence to calculate with these homotopy groups. 

\medskip

To deal with manifolds with non-trivial fundamental group, we recall that 
by Theorem \ref{thm:bordism} the classification of manifolds which carry \pscm s boils down to the computation of the subgroups $\Omega_n^{G,+}(B\pi)\subset \Omega_n^{G}(B\pi)$, for $G=\SO,\spin$ (and twisted versions thereof). This is difficult, since $\Omega^G_*$, the coefficient ring of the generalized homology theory $\Omega^G_*(\ )$ is large (rationally $\Omega^G_*$ is a polynomial ring generated by elements $x_4,x_8,\dots$, where the subscript indicates the degree of the generators). So it is desirable to replace the homology theories $\Omega^G_*(\ )$, $G=\spin,\SO$ by theories with a smaller coefficient rings, which the following result accomplishes.

Given a closed $n$\nb-manifold $M$, an orientation on $M$ determines a homology fundamental class $[M]^\SO\in H_n(M)$. Similarly, a spin structure on $M$ determines a $\ko$\nb-theory fundamental class $[M]^\Spin\in \ko_n(M)$, where $\ko_*(\ )$ is a generalized homology theory known as {\em connective real $K$\nb-homology}. For a topological space $X$, let $H_n^+(X)\subset H_n(X)$ be the subgroup consisting of homology classes which can written in the form $f_*[M]^\SO$, where $M$ is a closed oriented $n$\nb-manifold which carries a \pscm\ and $f\colon M\to X$ is a map. The subgroup $\ko_n^+(X)\subset \ko_n(X)$ is defined analogously, requiring a spin structure instead of an orientation on $M$ and replacing $f_*[M]^\SO$ by $f_*[M]^\spin$. 

\begin{thm}\label{thm:ko}
 Let $M$ be a closed connected $n$\nb-manifold, $n\ge 5$, with fundamental group $\pi$, and let  $u\colon M\to B\pi$  be the classifying map of the universal covering $\wt M\to M$. 
 \begin{enumerate}[label=\normalfont(\roman*),topsep=2pt,itemsep=-2pt]
\item If $M$ is spin then it carries a \pscm\ \Iff $u_*[M]^\Spin\in \ko_n^+(B\pi)$.
\item If $M$ is oriented, but $\wt M$ does not admit a spin structure, then $M$ carries a \pscm\ \Iff $u_*[M]^\SO\in H_n^+(B\pi)$. 
\end{enumerate}
\end{thm}

This was proved by Stolz \cite{St1994} ``localized at $2$'', using stable homotopy theory, and by  Rainer Jung (unpublished) and  F\"uhring \cite{Fu2013} with ``$2$ inverted'', using a geometric argument.  An outline of the proof is given in section \ref{ssec:H_ko}. 

The above theorem in particular implies that for a closed connected spin manifold $M$ of dimension $n\ge 5$ with fundamental group $\pi$ the vanishing of $u_*[M]^\spin\in \ko_n(B\pi)$ is a {\bf sufficient condition} for the existence of a \pscm\ on $M$. By Theorem \ref{thm:Ro} the vanishing of the index invariant $\alpha(M,u)\in \KO_n(C^*\pi)$ is a {\bf necessary condition} for such a metric. The index invariant $\alpha(M,u)$ is the image of $u_*[M]^\spin$ under the homomorphism 
\begin{equation}\label{eq:A}
\begin{tikzcd}
\ko_n(B\pi)\ar[r,"{p}"]&\KO_n(B\pi)\ar[r,"A"]&\KO_n(C^*\pi),
\end{tikzcd}
\end{equation}
where ${p}$ is the natural transformation relating connective  real $K$\nb-homology $\ko_n(\ )$ and periodic real $K$\nb-homology $\KO_n(\ )$, and $A$ is the map known as (real) {\em assembly map} (it is called the {\em Kasparov map} in \cite[p.\ 326]{Ro1986III}). 

In particular, if the composition \eqref{eq:A} is injective, then the vanishing of $\alpha(M,u)$ is a necessary and sufficient for the existence of a \pscm\ on closed connected spin $n$\nb-manifolds $M$, $n\ge 5$. The latter statement is commonly known as the {\em Gromov-Lawson-Rosenberg Conjecture} (Rosenberg conjectured this for finite $\pi$ \cite[Conjecture 0.1]{Ro1991}). Gromov and Lawson conjectured that the vanishing of $p(u_*[M]^\spin)\in \KO_n(B\pi)$ is a necessary and suffient condition for \pscm s \cite[p.\ 89]{GL1983} for ``reasonable'' $\pi$. 

\medskip

The GLR Conjecture has been proved for a number of groups. Concerning  finite groups $\pi$, it has been proved for odd order cyclic groups \cite{Ro1986II}, for $\Z/2$ \cite{RS1994}; more generally, for groups with periodic cohomology by \cite{KwSch1990} (of odd order) and \cite{BGS1997} (in general). It is sufficient to consider $p$\nb-groups, since a closed $n$\nb-manifold $M$, $n\ge 5$, with finite fundamental group $\pi$  carries a \pscm\ \Iff for every prime $p$ the covering of $M$ corresponding to the $p$\nb-Sylow subgroup of $\pi$ carries such a metric \cite[Prop.\ 1.5]{KwSch1990}. Botvinnik and Rosenberg proved the GLR Conjecture for spin $n$\nb-manifolds $M$ with fundamental group $(\Z/p)^r$ for $p$ odd, $n<r$ \cite[Thm.\ 2.3]{BR2005}, with some corrections of the proof provided by Hanke \cite[section 7]{Ha2016}. 

\medskip

Concerning infinite groups, the GLR Conjecture has been proved for free and free abelian groups, as well as for fundamental groups of orientable surfaces \cite{RS1994} (for these groups the assembly map is injective, and $B\pi$ splits stably as a wedge of sphere which implies that $p$ is injective as well). More generally, it has been verified  for groups $\pi$ for which $A$ is injective, the classifying space $B\pi$ is finite dimensional, and manifolds of dimension $n\ge \dim B\pi-4$ \cite[Thm.\ 2.1]{JS2000}. All infinite groups mentioned so far are torsion-free. The GLR-conjecture has also been proved some infinite groups with torsion, for example cocompact Fuchsian groups \cite{DP2003} and crystallographic groups of the form $\Z^n\rtimes\Z/p$, where $p$ is a prime and the $\Z/p$\nb-action on $\Z^n$ only fixes the origin \cite{DL2013}.

\medskip

Unfortunately, neither the GLR Conjecture nor the GL Conjecture  turn out to be true for a general group $\pi$: There is a $5$\nb-dimensional counterexample to the GLR conjecture with fundamental group $\pi=\Z^4\times \Z/3$ \cite{Sch1998}, and there is a $5$\nb-dimensional counterexample to the GL  conjecture with a torsion free fundamental group $\pi$ for which the assembly map is an isomorphism \cite{DSS2003}. These examples exhibit $5$\nb-dimensional spin manifolds for which all index obstructions vanish, but which do not carry \pscm s, since an obstruction coming from stable minimal hypersurfaces is non-zero (see Cor.\ \ref{cor:min_hyp} for the general statement, and \S\ref{ssec:nsconn} for a discussion of this counter example). 

The stable minimal hypersurface method does not produce any restrictions for the existence of \pscm s on manifolds with {\em finite} fundamental group, and no counter examples to the GLR Conjecture for finite fundamental groups are known. For general $\pi$, there is currently no conjectural characterization of the subgroup $\ko_n^+(B\pi)$.

\medskip

The only known obstructions for the existence of a \pscm\ on a closed manifolds $M$ of dimension $\ge 5$ come from index theory or the stable minimal hypersurface method. Index theory is not known to produce obstructions if the universal cover $\wt M$ is non-spin, and the stable minimal hypersurface method does not produce obstructions if the fundamental group is finite. So the following is the optimist's conjecture.

\begin{conj}\label{conj:nonspin_cover} Let $M$ be a connected closed manifold of dimension $n\ge 5$ with finite fundamental group $\pi$ whose universal cover does not admit a spin structure. Then $M$ carries a \pscm.
\end{conj}

It suffices to consider  $p$\nb-groups, since $M$ carries a \pscm\ \Iff\ the coverings of $M$ whose fundamental groups are the $p$\nb-Sylow subgroups of $\pi$ carry \pscm s by \cite[Prop.\ 1.5]{KwSch1990}. Conjecture \ref{conj:nonspin_cover} is known for many cases of abelian $p$\nb-groups, for example for elementary abelian $2$\nb-groups $\pi=(\Z/2)^r$ 
by combining work of Joachim \cite{Jo2004} and Botvinnik-Rosenberg \cite{BR2005}. 
For elementary abelian $p$\nb-groups $\pi=(\Z/p)^r$ with $p$ odd, it was proved for $n>r$ \cite[Thm.\ 2.3 and Thm.\ 2.4]{BR2005} (with later corrections of the proof provided by \cite{Ha2016}). In fact, they prove a stronger statement, replacing the dimension restriction $n>r$ by the assumption that the element $u_*[M]^\SO\in H_n(B\pi)$ is {\em $p$-atoral}, i.e., for any collection of elements $\alpha_1,\dots, \alpha_n\in H^1(B\pi;\Z/p^\ell)$ for any $\ell>0$ the Kronecker product $\bra{\alpha_1\cup\dots\cup\alpha_n,u_*[M]^\SO}\in \Z/p^\ell$ vanishes. More recently, Conjecture \ref{conj:nonspin_cover} has been proved by Hanke for abelian $p$\nb-groups, $p$ odd, assuming that $u_*[M]^\SO\in H_n(B\pi)$ is $p$\nb-atoral \cite{Ha2019}.

\medskip

The author wishes to thank Jonathan Rosenberg and Bernhard Hanke for bringing important references to the author's attention. This work was partially supported by NSF grant DMS-1547292.
\section{Obstructions to positive scalar curvature metrics}
This section gives a quick overview of the obstructions to \pscm s derived from considering the Dirac operator, which was pioneered by Lichnerowicz in the early 1960's \cite{Li1963}, further developed by Hitchin \cite{Hi1974} and Rosenberg \cite{Ro1986III}. The last part of this section describes additional obstructions coming from the {\em stable minimal hypersurface method} developed by Schoen and Yau \cite{SY1979a}, \cite{SY1979b}.

\subsection{The Dirac operator}\label{ssec:Dirac}

The Dirac operator, constructed in this generality by Atiyah and Singer in the context of their proof of the Index Theorem, is a first order elliptic differential operator $D$ acting on the space $\Gamma(S)$ of smooth sections of a vector bundle $S\to M$ called the {\em spinor bundle} of a Riemannian manifold $M$. The construction of the spinor bundle requires that $M$ comes equipped with a {\em spin structure} (see below). The Dirac operator is closely related to the scalar curvature function $s\in C^\infty(M)$ via the  {\em Lichnerowicz formula} \cite[Ch II, Thm.\ 8.8]{LM1989}
\begin{equation}\label{eq:Li}
D^2=\nabla^*\nabla +\frac 14 s.
\end{equation}
Here $\nabla\colon \Gamma(S)\to \Gamma(T^*M\otimes S)$  is the connection on the spinor bundle $S$, induced by the Levi-Civita connection on the tangent bundle $TM$, and $\nabla^*\colon \Gamma(T^*M\otimes S)\to \Gamma( S)$ is the adjoint of $\nabla$ with respect to the inner product $\bra{\ ,\ }$ on the space of sections $\Gamma(S)$ determined by the Riemannian metric on $M$ and the induced inner product on the fibers of $S$. Using the fact that $D$ is self-adjoint, this implies the following inequality for $\psi\in \Gamma(S)$:
\[
||D\psi||^2=\bra{D\psi,D\psi}=\bra{D^2\psi,\psi}=\bra{\nabla^*\nabla \psi+s\psi,\psi}
=||\nabla\psi||^2+\bra{s\psi,\psi}\ge \bra{s\psi,\psi}.
\]
It follows that if the scalar curvature function $s$ is strictly positive on the closed manifold $M$, then the kernel of $D$ is trivial. 

As explained below, the spinor bundle $S$ is $\Z/2$\nb-graded, i.e., $S=S^+\oplus S^-$, and with respect to the induced decomposition of the space of sections $\Gamma(S)=\Gamma(S^+)\oplus \Gamma(S^-)$, the Dirac operator $D$ has the form
\begin{equation}\label{eq:D_odd}
D=\left(\begin{matrix} 0&D^-\\D^+&0\end{matrix}\right)\colon 
\Gamma(S^+)\oplus \Gamma(S^-)\ra \Gamma(S^+)\oplus \Gamma(S^-).
\end{equation}
Hence the kernel of $D$ decomposes in the form $\ker D=\ker D^+\oplus \ker D^-$. The fact that $D$ is an elliptic operator implies that $\ker D$ is finite dimensional. Moreover, $D$ is self-adjoint, which implies that $D^-$ is the adjoint  of $D^+$, and hence the {\em index of $D^+$}, defined by
\[
\operatorname{index}(D^+):=\dim \ker D^+-\dim \coker D^+
\]
is equal to $\sdim\ker D:=\dim \ker D^+-\dim \ker D^-$, the {\em superdimension}  of the $\Z/2$\nb-graded vector space $\ker  D=\ker D^+\oplus \ker D^-$.

Unlike the dimensions of $\ker D^+$ and $\ker D^-$ which in general depend on the {\em geometry}, i.e., the Riemannian metric used in the construction of the Dirac operator, the index of $D^+$ depends only on the {\em topology}, i.e., on the closed manifold $M$. In fact, according to the Atiyah-Singer Index Theorem, the index of $D^+$ can be identified with a topological invariant $\wh A(M)$, known as Hirzebruch's $\wh A$\nb-genus \cite[Ch.\ III, Thm.\ 13.10]{LM1989}. This is  a characteristic number, that is, a number obtained by evaluating a particular polynomial in the Pontryagin classes of the tangent bundle $TM$ on the fundamental homology class of $M$. 

Summarizing, the Lichnerowicz formula \eqref{eq:Li} together with the Atiyah-Singer index theorem for the Dirac operator imply the following result \cite{Li1963}.

\begin{thm}{\bf (Lichnerowicz).} Let $M$ be a closed spin manifold of dimension $n=4k$ which admits a \pscm. Then the $\wh A$\nb-genus $\wh A(M)$ vanishes. 
\end{thm}

\begin{ex}\label{ex:hyp} The degree $d$ hypersurface  $X^2(d)=\{[z_0,z_1,z_2,z_3]\in \CP^3\mid z_0^d+z_1^d+z_2^d+z_3^d=0\}$ in the complex projective space $\CP^3$ is a simply connected closed manifold of real dimension $4$. Its $\wh A$\nb-hat genus is $\wh A(X^2(d))=\frac {d(d-2)(d+2)}{24}$ \cite[Ch.\ IV, formula 4.4]{LM1989}, and hence $X^2(d)$ does not admit a \pscm\ if $d\ge 4$ and $d$ is even ($d$ even guarantees that $X^2(d)$ has a spin structure). 
\end{ex}

\subsection{The Clifford linear Dirac operator}\label{ssec:Cliff_Dirac}
There is an important refinement of Lichnerowicz' theorem due to Hitchin \cite[\S 4]{Hi1974}. Following \cite[Ch.\ II, section 7]{LM1989} this can be obtained using a variant of the Dirac operator called the {\em Clifford linear Dirac operator}. The {\em Clifford algebra} $\Cl_n$ is the quotient of the tensor algebra generated by $\R^n$ modulo the ideal generated by the elements of the form $v\otimes v+||v||^2$ for $v\in \R^n$. The natural $\Z$\nb-grading on the tensor algebra induces a $\Z/2$\nb-grading $\Cl_n=\Cl_n^+\oplus \Cl_n^-$ on the Clifford algebra. An extremely useful property  of the Clifford algebra is that $\Spin(n)$, a non-trivial double covering group of $SO(n)$, can be constructed as a subgroup of the group of units in $\Cl_n$ (namely $\Spin(n)$ is the intersection of the subgroup generated by unit vectors in $\R^n$ and $\Cl_n^+$) \cite[Ch I, Thm.\ 2.9]{LM1989}. 

To describe the Clifford linear Dirac operator, we first need to be more explicit about the construction of the spinor bundle.
Let $M$ be an oriented Riemannian $n$\nb-manifold and let $SO(M)\to M$ be the oriented orthonormal frame bundle of $M$; its fiber $SO(M)_x$ over a point $x\in M$ consists of all orientation preserving isometries $f\colon \R^n \overset\cong\ra T_xM$. 

\begin{defn} A {\em spin structure}  on $M$ (see e.g.\ \cite[Ch.\ II, Def.\ 1.3]{LM1989}) is a principal $\Spin(n)$\nb-bundle $\Spin(M)\to M$ equipped with a double covering map $q\colon \Spin(M)\to  SO(M)$ which is $\Spin(n)$\nb-equivariant (with $\Spin(n)$ acting on $\SO(M)$ via the projection map $p\colon \Spin(n)\to \SO(n)$), and which  makes the diagram 
\[
\begin{tikzcd}
\Spin(M)\ar[rr,"q"]\ar[rd]&&SO(M)\ar[dl]\\
&M&
\end{tikzcd}
\]
commutative. 
\end{defn}

\begin{construction} {(\bf Spinor bundle.)} Let $M$ be a Riemannian spin $n$\nb-manifold. Then the {\em spinor bundle} on $M$ is a $\Z/2$\nb-graded vector bundle associated to the principal $\Spin(n)$\nb-bundle $\Spin(M)\to M$ provided by the spin structure on $M$. More precisely, there is a spinor bundle $S_\Delta$ associated to any $\Z/2$\nb-graded left module $\Delta$ over the Clifford algebra $\Cl_n$. Since $\Spin(n)\subset \Cl_n^\times$ is a subgroup of the group of invertible elements of $\Cl_n$, the module $\Delta$ is in particular a representation of $\Spin(n)$, and 
\[
S_\Delta:=\Spin(M)\times_{\Spin(n)}\Delta\ \ra\ M
\]
is the associated vector bundle. The $\Z/2$\nb-grading of $\Delta=\Delta^+\oplus \Delta^-$ is preserved by the $\Spin(n)$\nb-action, and hence it induces a $\Z/2$\nb-grading $S_\Delta=S_\Delta^+\oplus S_\Delta^-$ on the spinor bundle $S_\Delta$. 

There is a conventional choice of an irreducible $\Z/2$\nb-graded module over $\Cl_n$ (there are one or two such modules, depending on $n\mod 4$); the corresponding spinor bundle is the bundle we denoted by $S$ in the previous section.
\end{construction}

For {\em any} choice of the $\Cl_n$\nb-module $\Delta$ there is a Dirac operator $D^M_\Delta\colon \Gamma(S_\Delta)\to \Gamma(S_\Delta)$ with the same formal properties of the Dirac operator $D$ discussed in the previous section. In particular, $D^M_\Delta$ is self-adjoint, and it is invertible if the Riemannian metric used in its construction has positive scalar curvature. 

At first glance choosing the left $\Cl_n$\nb-module $\Delta$ to be $\Cl_n$ itself might seem silly, since $\Delta$ decomposes as a sum of irreducible modules, and so choosing an irreducible module seems less redundant. However, the spinor bundle $S_{\Cl_n}=\Spin(M)\times_{\Spin(n)}\Cl_n$ has more structure: $\Cl_n$ is a $\Cl_n$\nb-bimodule: its left module structure is used for the construction of $S_{\Cl_n}$, while its right module structure gives each fiber of $S_{\Cl_n}$ the structure of a $\Z/2$\nb-graded right module. Hence the space $\Gamma(S_{\Cl_n})$ of smooth sections is a right $\Cl_n$\nb-module. For that reason, the vector bundle $S_{\Cl_n}$ is called the {\em $\Cl_n$\nb-linear spinor bundle}. The Dirac operator
\begin{equation}\label{eq:Cl_linear}
D^M_{\Cl_n}\colon \Gamma(S_{\Cl_n})\ra \Gamma(S_{\Cl_n})
\end{equation}
commutes with the right $\Cl_n$\nb-action, and hence $D_{\Cl_n}$ is known as the {\em $\Cl_n$\nb-linear Dirac operator}. 

To describe the index invariant $\alpha(M)\in \KO_n$, we first recall that the real $K$\nb-theory group $\KO_n=\KO_n(\point)=\KO^{-n}(\point)$ can be described in terms of modules over Clifford algebras. More precisely, Atiyah, Bott and Sharpiro constructed an isomorphism \cite[Thm.\ 11.5]{ABS1963}, \cite[Ch.\ I, Thm.\ 9.27]{LM1989}
\begin{equation}\label{eq:KO_Cliff}
\fM(\Cl_n)/i^*\fM(\Cl_{n+1})\cong\KO_n.
\end{equation}
Here $\fM(A)$ denotes for a $\Z/2$\nb-graded algebra $A$ the group completion of the semi-group of isomorphism classes of $\Z/2$\nb-graded finitely generated $A$\nb-modules (with the semi-group structure provided by the direct sum of modules). The homomorphism 
\begin{equation}\label{eq:i^*}
i^*\colon \fM(\Cl_{n+1})\to \fM(\Cl_{n})
\end{equation}
 is induced by the inclusion map $i\colon \Cl_n\into \Cl_{n+1}$. We remark that we do not need to specify whether $\fM(\Cl_n)$ consists of left-module or right-modules, since there is an anti-involution $\beta\colon \Cl_n\to \Cl_n$ (i.e., $\beta(aa')=\beta(a')\beta(a)$, without signs), determined by $\beta(v)=v$ for $v\in \R^n\subset\Cl_n$. This anti-involution allows us to turn a right module into a left module and vice versa. 

The $\Cl_n$\nb-linearity of the operator \eqref{eq:Cl_linear} implies that its kernel $\ker D^M_{\Cl_n}$ is a $\Z/2$\nb-graded $\Cl_n$\nb-module. If the manifold $M$ is closed, the fact that $D^M_{\Cl_n}$ is an elliptic differential operator implies that its kernel is finite dimensional, and hence $\ker D^M_{\Cl_n}\in \fM(\Cl_n)$. The index invariant $\alpha(M)$ is defined by
\begin{equation}\label{eq:alpha}
\alpha(M):=[\ker D^M_{\Cl_n}]\in \fM(\Cl_n)/i^*\fM(\Cl_{n+1})\cong \KO_n.
\end{equation}
If the Riemannian metric $g$ used in the construction of the Dirac operator $D^M_{\Cl_n}$ has \psc, then by the Lichnerowicz formula \eqref{eq:Li} $\ker D^M_{\Cl_n}$ is trivial, and hence $\alpha(M)=0$. This proves the Lichnerowicz-Hitchin Theorem \ref{thm:LH}. 

\medskip

We note that the construction of the Clifford linear Dirac operator $D^M_{\Cl_n}$ involves the Riemannian metric $g$ on $M$ and  hence $\ker D^M_{\Cl_n}$ depends on $g$. For a path of metrics $g_t$, there is a corresponding path of Clifford linear Dirac operators $D_t$. The dimension of $\ker D_t$ might jump at a value $t_0$ due to an eigenvalue $\lambda_t$ of $D_t$ approaching $0$ for $t\to t_0$.

In the rest of this section we explain why $[\ker D^M_{\Cl_n}]\in \fM(\Cl_n)/i^*\fM(\Cl_{n+1})$ {\em does not} depend on the metric, i.e., why it is an invariant of the smooth manifold $M$. 
Let $E_\lambda$ be the finite dimensional eigenspace of $D=D^M_{\Cl_n}$ with eigenvalue $\lambda\in \R\setminus \{0\}$ (all eigenvalues are real since $D$ is self-adjoint). Let $\tau$ be the grading involution of the $\Z/2$\nb-graded vector space $\Gamma(S_{\Cl_n})$; it anti-commutes with $D$, i.e., $\tau D=-D\tau$. This is equivalent to the fact that  in the matrix decomposition \eqref{eq:D_odd} of $D$ with respect to the grading decomposition $\Gamma(S_{\Cl_n})=\Gamma(S^+_{\Cl_n})\oplus \Gamma(S^-_{\Cl_n})$ the diagonal entries vanish. In particular, for $\psi\in E_\lambda$
\[
D(\tau(\psi))=-\tau(D(\psi))=-\tau(\lambda \psi)=-\lambda\tau(\psi),
\]
and hence $\tau(\psi)\in E_{-\lambda}$. The Clifford linearity of $D$ implies that $E_\lambda\oplus E_{-\lambda}$ is a $\Z/2$\nb-graded module over $\Cl_n$. We claim that it is in the image of the homomorphism \eqref{eq:i^*}, i.e., that the $\Cl_n$\nb-module structure on $E_\lambda\oplus E_{-\lambda}$ can be extended to a $\Cl_{n+1}$\nb-module structure. It suffices to construct an endomorphism $e_{n+1}$ of $E_\lambda\oplus E_{-\lambda}$ satisfying $e_{n+1}^2=-1$ and $e_{n+1}v=-ve_n$ for $v\in \R^n\subset \Cl_n$. It is easy to check that $e_{n+1}:=\frac 1{|\lambda|}D\tau$ has these properties. 

This show that $\ker D^M_{\Cl_n}$ represents the same class in $\fM(\Cl_n)/i^*\fM(\Cl_{n+1})$ as 
\begin{equation}\label{eq:cutoff}
\bigoplus_{|\lambda|<\mu}E_\lambda
=\ker D^M_{\Cl_n}\oplus\bigoplus_{0<\lambda<\mu}E_{-\lambda}\oplus E_\lambda
\end{equation}
for any $\mu>0$. If $\mu$ is  not an eigenvalue of $D^M_{\Cl_n}$ constructed using a metric $g$, then $\mu$ is also not an eigenvalue for those operators obtained from metrics in an open neighborhood of $g$ in the space of metrics. This implies that the isomorphism class of the $\Cl_n$\nb-module \eqref{eq:cutoff} is constant in that neighborhood and shows that its class in $\fM(\Cl_n)/i^*\fM(\Cl_{n+1})$ is in fact independent of the metric. 
\subsection{Obstructions associated to coverings}\label{ssec:obs_cov}
In this section, let $M$ be a closed connected spin manifold of dimension $n$. Let $\pi$ be a discrete group, let $\wt M\to M$ be a principal $\pi$\nb-bundle and let $f\colon M\to B\pi$ be its classifying map. We are particularly interested in the case where $\pi$ is the fundamental group of $M$ and $\wt M\to M$ is the universal covering, but it is useful to work in this generality. Let $D_{\Cl_n}^{\wt M}$ be the Clifford linear Dirac operator on $\wt M$.  This operator commutes with the action of $\Cl_n$ and the action of $\pi$ by deck transformations on $\wt M$, and hence with the action of $\Cl_n\otimes\R\pi$, where $\R\pi$ is the real group ring of $\pi$. In particular, $\ker D_{\Cl_n}^{\wt M}$ is a module over $\Cl_n\otimes\R\pi$. 

Assuming that $\pi$ is {\em finite}, the manifold $\wt M$ is compact, and hence $\ker D_{\Cl_n}^{\wt M}$ is a finite dimensional vector space. The argument in the previous section implies that the class
\[
\alpha(M,f):=[\ker  D_{\Cl_n}^{\wt M}]\in \fM(\Cl_n\otimes\R\pi)/i^* \fM(\Cl_{n+1}\otimes\R\pi)
\]
depends only on the manifold $M$, not the metric on $M$, and only on the homotopy class of the map $f\colon M\to B\pi$. The $K$\nb-theory group $\KO_n(\R\pi)$ is equal to $\fM(\Cl_n\otimes\R\pi)/i^* \fM(\Cl_{n+1}\otimes\R\pi)$ by definition.

For infinite $\pi$, the construction involves more analytic details. A central role is played by the flat vector bundle
\[
\wt M\times_\pi C^*\pi\ra M,
\]
where $C^*\pi\supset \R\pi$ is the {\em reduced group $C^*$\nb-algebra of $\pi$}, a suitable completion of $\R\pi$. The dimension of this vector bundle is the cardinality of $\pi$; in particular, it is infinite dimensional if $\pi$ is infinite. Fortunately, it can also be thought of as a {\em line bundle}, if we don't regard it as a real vector bundle, but rather as a bundle of right $C^*\pi$\nb-modules. The Clifford linear spinor bundle on $M$ can be tensored with this flat ``line bundle'' to obtain a vector bundle whose sections are modules over the $C^*$\nb-algebra $\Cl_n\otimes C^*\pi$. Moreover, there is a twisted Dirac operator acting on this section 
space which is $\Cl_n\otimes C^*\pi$\nb-linear. Using Kasparov's description of $\KO_n(A)$ for real $C^*$\nb-algebras $A$, and the index theory of Fredholm operators acting on Hilbert $A$\nb-modules developed by Miscenko-Fomenko, Rosenberg extracts the index of this twisted Dirac operator as an element $\alpha(M,f)\in \KO_n(C^*\pi)$ \cite[section 3]{Ro1986III}. 
\subsection{Obstructions coming from stable minimal hypersurfaces}
\label{ssec:min_hyp}

The following result was proved by Schoen-Yau using stable minimal hypersurfaces \cite[Proof of Thm.\ 1]{SY1979a}, \cite[Thm.\ 5.1]{SY1979b}.

\begin{thm}\label{thm:min_hyp}{\bf (Schoen and Yau).}
Let $M$ be a closed Riemannian manifold of dimension $n\ge 3$ equipped with a \pscm. If $N\subset M$ is a smooth hypersurface with trivial normal bundle, and if $N$ is a local minimum of the volume functional,  then $N$ carries a \pscm. 
\end{thm}

Results from geometric measure theory guarantee that if $M$ is orientable of dimension $n\le 7$, and $x\in H_{n-1}(M)$ is a non-vanishing integral homology class, then there is a smooth orientable hypersurface $N\subset M$ which represents $x$ and is a local minimum of the volume functional (see \cite[Thm.\ 1.3]{Sch1998} for relevant references on geometric measure theory). Without the dimension restriction $n\ge 7$, any homology class can be represented by an area-minimizing current $N$ whose singularities have codimension $\ge 7$ in $N$. Hence the condition $n\ge 7$ guarantees that $N$ has no singularities. 

For a topological space $X$, let $H_n^+(X)$ be the subgroup of $H_n(X)$ consisting of homology classes  of the form $f_*[M]$, where $M$ is a closed oriented $n$\nb-manifold that carries a \pscm, $[M]\in H_n(M)$ is the fundamental homology class of $M$, and $f\colon M\to X$ is a map. The author observed the following consequence of the results discussed above (see \cite[Cor.\ 1.5]{Sch1998}).

\begin{cor}\label{cor:min_hyp}
For a topological space $X$, let $\alpha\cap\colon H_n(X)\to H_{n-1}(X)$ be the map given by the cap product with a cohomology class $\alpha\in H^1(X)$ (all (co)homology groups here have integer coefficients). Then for $3\le n\le 7$ this homomorphism maps $H_n^+(X)$ to $H^+_{n-1}(X)$. 
\end{cor}

\begin{proof} If $x\in H_n(X)$ is represented by $f\colon M\to X$, then $\alpha\cap x\in H_{n-1}(X)$ is represented by $f_{|N}\colon H\to M$, where $N\subset M$ is a hypersurface in $M$ representing the homology class $f^*\alpha\cap [M]\in H_{n-1}(M)$. If $x$ belongs to $  H_n^+(X)\subset H_n(X)$ we can assume that $M$ carries a \pscm\ and thanks to the dimension restriction $n\le 7$, we can assume that the hypersurface $N\subset M$ is a local minimizer of the volume function. Hence $N$ carries a \pscm\ by Theorem \ref{thm:min_hyp}, and so $f_{|N}\colon N\to X$ represents a homology class in $H_{n-1}^+(X)$.
\end{proof}

\section{Constructions of \pscm s}

\subsection{Basic examples of \pscm s}
There are many manifolds which carry \pscm s. For example, the  standard Riemannian metric  on the sphere $S^n$ (known as ``round metric'') has \psc\ for $n\ge 2$. The complex projective space $\CP^n$ is the quotient of the sphere $S^{2n+1}\subset \C^{n+1}$ with respect to the $S^1$\nb-action given by scalar multiplication. Since $S^1$ acts by isometries, the round metric on $S^{2n+1}$ induces a metric on $\CP^n$ known as {\em Fubini-Study metric} (characterized by the property that the projection map $S^{2n+1}\to \CP^n$ is a Riemannian submersion). The Fubini-Study metric on $\CP^n$ has \psc. The same construction yields a \pscm\ on the real projective space $\RP^n$, $n\ge 2$ (the quotient of an isometric $\Z/2$\nb-action on $S^n$) and the quaternionic projective space $\HP^n$ (the quotient of an isometric $\SU(2)$\nb-action on $S^{4n+1}\subset \H^{n+1}$).  

The following observation is elementary, but very useful to find many examples of manifolds that carry \pscm s. 

\begin{nnumber}{\bf Observation.}\label{ob:fib} Let $g$ be a \pscm\ on a closed manifold $M$. Then the following manifolds admit metrics of \psc:
\begin{enumerate}[topsep=2pt,itemsep=-2pt]
\item The product $M\times N$ for any closed manifold $N$.
\item The total space of any fiber bundle $E\to N$ with fiber $M$ whose structure group is the isometry group of $(M,g)$; i.e., there are local trivializations such that all transition functions are isometries of $(M,g)$. 
\end{enumerate}
\end{nnumber}
To prove the first claim, pick a metric  $h$ on $N$ (not necessarily with positive scalar curvature) and consider the product metric $g\times h$ on $M\times N$. The scalar curvature at a point $(x,y)\in M\times N$ with respect to the metric $g\times h$ is given by the formula
\[
s((x,y);g\times h)=s(x;g)+s(y;h).
\]
So without further assumptions on $h$ there is no reason for the scalar curvature of the product metric $g\times h$ to be positive. The trick is to ``shrink $M$'', i.e., to replace $g$ by $tg$ for $t>0$, and letting $t$ approach $0$. Then
\[
s((x,y);tg\times h)=s(x;tg)+s(y;h)=\frac 1ts(x;g)+s(y;h),
\]
and hence this is {\em positive} for $t$ sufficiently small thanks to $s(x;g)>0$. Compactness of  $M$ and $N$ guarantee that there is one $t$ that will do the job for all $(x,y)\in M\times N$. 

In the case of a fiber bundle $E\to N$ with fiber $M$, the construction of a metric on $E$ requires the choice of a metric $h$ on $N$, and a connection on the fiber bundle. The condition on the structure group guarantees that the connection can be chosen such that the fibers are totally geodesic. Using the O'Neill formulas, which express the curvature of $E$ in terms of the connection and the curvatures of $M$ and $N$, then shows that the ``shrinking trick" still works for such a fiber bundle: shrinking the fibers $M$ sufficiently makes the scalar curvature function on the total space $E$ everywhere positive. 
\subsection{Surgery results}
Surgery is a basic technique to modify the topology of a manifold $M$ by removing a piece of $M$ and replacing it by another piece. The simplest example is the surgery that replaces the disjoint union of two $n$\nb-manifolds $M_1$, $M_2$ by their connected sum $M_1\# M_2$. Independently Gromov-Lawson \cite[Thm.\ A]{GL1980} and Schoen-Yau \cite[Cor.\ 6]{SY1979a} proved the following result.

\begin{thm}[\bf Surgery Theorem] \label{thm:surgery} Let $M$ be a manifold that carries a \pscm. Then any manifold obtained by a surgery of codimension $\ge 3$ also carries a \pscm. 
\end{thm}

This result implies in particular that if two manifolds $M_1$, $M_2$ of dimension $\ge 3$ carry \pscm s, then so does their connected sum $M_1\# M_2$. We first outline the proof 
for this special case. Then we recall what a $k$\nb-surgery is, and finally we outline the proof of the Surgery Theorem in the general case. 

\medskip

Let $M_1$, $M_2$ be closed connected manifolds of dimension $n$. The {\em connected sum} $M_1\#M_2$ is the closed connected $n$\nb-manifold obtained in the following way:  Deleting an open $n$\nb-ball $\mathring D^n$ from $M_i$ gives a manifold $M_i\setminus \mathring D^n$ with boundary $S^{n-1}$. Identifying a point on the boundary of $M_1\setminus \mathring D^n$ with the corresponding point on the boundary of $M_2\setminus \mathring D^n$ yields the connected sum
\begin{equation}\label{eq:conn_sum}
M_1\# M_2:=(M_1\setminus \mathring D^n)\cup_{S^{n-1}}(M_2\setminus \mathring D^n).
\end{equation}
Let $g_1$, $g_2$ be \pscm s on $M_1$ resp.\ $M_2$. The key for obtaining a \pscm\ on the connected sum $M_1\# M_2$ is to modify the metric $g_i$ to obtain a  \pscm\ $g_i'$ on the manifold $M_i\setminus \mathring D^n$ which
\begin{enumerate}[label=(\roman*),topsep=2pt,itemsep=-4pt]
\item restricts to the standard round metric on the boundary $\p (M_i\setminus \mathring D^n)=S^{n-1}$, and
\item is a product metric near the boundary. 
\end{enumerate}
These conditions guarantee that the metrics $g_1'$ and $g_2'$ fit together to give a \pscm\ on the connected sum $M_1\# M_2$. We remark that the dimension restriction $n\ge 3$ ensures that the standard metric on $S^{n-1}$ has \psc. 

Let $g$ be \pscm\ on a manifold $M$ of dimension $n\ge 3$ and fix a point $p\in M$. The modified \pscm\ $g'$ with properties (i) and (ii) on the complement of a disk $D^n\subset M$ containing $p$ is constructed as follows. Think of $M\setminus \mathring D^n$ as a hypersurface $N$ in the product $M\times [0,1]$ as shown in the following picture:
\begin{equation*}
\begin{tikzpicture}[scale=.6]
\fill[fill=black!10](-3,4)--(3,6)--(3,-4)--(-3,-6)--cycle;
\fill[fill=black!10](0,4)..controls (0,2) and(2,.5)..(8,.5) arc(90:270:.25 and .5)
--(7,-.5)..controls (2,-.5) and (0,-2)..(0,-4)--cycle;
\draw[thick](0,4)..controls (0,2) and(2,.5)..(7,.5)--(8,.5);
\draw[thick](0,-4)..controls (0,-2) and(2,-.5)..(7,-.5)--(8,-.5);
\draw[thick](8,0)circle (.25 and .5);
\filldraw (0,0) circle(.1)node[below]{$(p,0)$};
\node[label={[rotate=15]right: $M\times\{0\}$}] at (-3.2,3.5){};
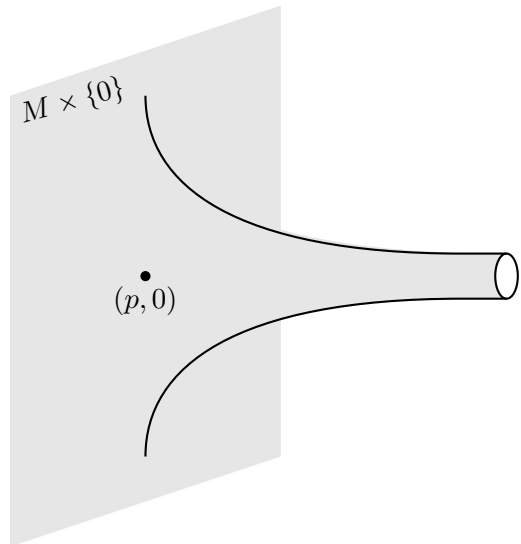
\node at (2,-7){The hypersurface $N\subset M\times [0,1]$};
\end{tikzpicture}
\end{equation*}
More precisely, the hypersurface $N$ is determined by choosing a smooth curve $\gamma$ in the $(t,r)$\nb-plane as shown below: 
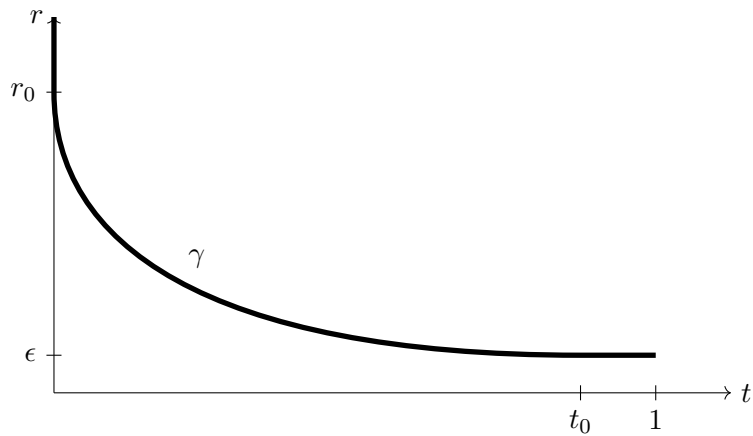
\begin{equation*}\label{eq:curve}
\begin{tikzpicture}
\draw[->] (0,0)--(9,0)node[right]{$t$};
\draw[->] (0,0)--(0,5)node[left]{$r$};
\draw(8,-.1)node[below]{$1$}--+(0,.2);
\draw(7,-.1)node[below]{$t_0$}--+(0,.2);
\draw(-.1,4)node[left]{$r_0$}--+(.2,0);
\draw(-.1,.5)node[left]{$\epsilon$}--+(.2,0);
\draw[line width=2pt](0,5)--(0,4)..controls (0,2) and(2,.5)..node[above right]{$\gamma$}(7,.5)--(8,.5);
\end{tikzpicture}
\end{equation*}
The curve contains the points on the positive $r$\nb-axis for $r>r_0$, and the horizontal line segment $\{(t,\epsilon)\mid t_0\le t\le 1\}$. 
Such a curve determines a hypersurface $N$ of the product $M\times [0,1]$, defined by 
\[
N=\{(x,t)\in M\times[0,1]\mid (\dist(x,p),t)\in \gamma\},
\]
where $\dist(x,p)$ is the distance between the points $x$ and  $p$. Choosing $r_0$ to be the injectivity radius of $g$ guarantees that $N$ is a smooth hypersurface of $M\times[0,1]$. The product metric on $M\times[0,1]$ induces a Riemannian metric on $N$, such that the  boundary $\p N$ is isometric to the ambient sphere $S^{n-1}_\epsilon(p)$ of radius $\epsilon$ around the point $p\in M$; moreover, the metric is a product metric near the boundary.  

The key point is that a careful choice of the curve $\gamma$ guarantees that the positive contributions to the scalar curvature of $N$ coming from the directions tangent to the $(n-1)$\nb-spheres,  given by the intersection of $N$ with $M\times \{t\}$ for fixed  $t\in (0,1]$,  dominates the negative contribution coming from the orthogonal direction. After an additional deformation of the metric in a small neighborhood of $\p N$, this is  a \pscm\ on $N$ which is the standard round metric on $\p N=S^{n-1}$, as well as a product metric near the boundary. 

\medskip

Next we recall the what a $k$\nb-surgery is, and outline the proof of the Surgery Theorem \ref{thm:surgery}. Let $M$ be a manifold of dimension $n$ and assume that there is an embedding of $S^k\times D^{n-k}$ in $M$. Then the complement $M\setminus (S^k\times \mathring D^{n-k})$ is a manifold with boundary $S^k\times S^{n-k-1}$, which can be glued with $D^{k+1}\times S^{n-k-1}$ along their common boundary to obtain a new $n$\nb-manifold 
\begin{equation}\label{eq:surgery}
\wh M:=M\setminus (S^k\times \mathring D^{n-k})
\cup_{S^k\times S^{n-k-1}}(D^{k+1}\times S^{n-k-1}).
\end{equation}
We say that $\wh M$ is obtained from $M$ by a {\em $k$\nb-surgery}  (or by a {\em surgery of codimension $n-k$}). For example, the connected sum $M_1\# M_2$ of two $n$\nb-manifolds is obtained by a $0$\nb-surgery from the disjoint union $M_1\amalg M_2$. 

The proof of the Surgery Theorem follows the same strategy as in the connected sum case: the \pscm\ on $\wh M$ is obtained by constructing \pscm s on the two pieces, $D^{k+1}\times S^{n-k-1}$ and $M\setminus (S^k\times \mathring D^{n-k})$, which agree on the common boundary $S^k\times S^{n-k-1}$, and which are product metrics near the boundary. On $S^k\times \mathring D^{n-k}$, this is the product metric of the standard round metric on $S^k$, and the hemisphere metric on $D^{n-k}\subset S^{n-k}$ (slightly deformed in order to make it a product metric near the boundary). On $M\setminus (S^k\times \mathring D^{n-k})$, it is a modification $g'$ of the given \pscm\ $g$ on $M$. It is obtained by generalizing the technique used above. The curve $\gamma$ as above determines a hypersurface $N\subset M\times [0,1]$ given by
\[
N=\left\{(x,t)\in M\times[0,1]\mid (\dist(x,S^k),t)\in \gamma\right\},
\]
where $\dist(x,S^k)$ is distance of $x$ from the embedded $k$\nb-sphere $S^k\subset S^k\times D^{n-k}\subset M$ that we do surgery on. The intersection of $N$ with $M\times\{t\}$ for  fixed $t\in (0,1]$ consists of the points in $M$ that have distance $\gamma(t)$ from $S^k$; this is a fiber bundle over $S^k$ whose fibers are spheres of dimension $n-k-1$. So the codimension restriction $n-k\ge 3$ guarantees that the curvature in the direction tangent to these sphere has a positive contribution to the scalar curvature of the hypersurface. As in the special case, a careful choice of $\gamma$ guarantees that the scalar curvature of $N$ is positive, and an additional modification of the metric near the boundary ensures that its  restriction  to $\p N=S^k\times S^{n-k-1}$ is the standard product metric.

\subsection{Bordism results}
The goal of this section is to outline how the Surgery Theorem \ref{thm:surgery} is used to prove the Bordism Theorem \ref{thm:bordism}, which we restate here for the convenience of the reader. 

\begin{thm}{\bf (Bordism Theorem).}\label{thm:bordism2}
 Let $M$ be a closed connected $n$\nb-manifold, $n\ge 5$, with fundamental group $\pi$, and let  $u\colon M\to B\pi$  be the classifying map of the universal covering $\wt M\to M$. 
\begin{enumerate}[label=\normalfont(\roman*),topsep=2pt,itemsep=-2pt]
\item If $M$ is spin, then it carries a \pscm\ \Iff  $[M,u]\in \Omega_n^{\spin,+}(B\pi)$.
\item If $M$ is oriented and $\wt M$ is non-spin, then $M$ admits a \pscm\ \Iff $[M,u]\in \Omega_n^{\SO,+}(B\pi)$. 
\end{enumerate}
\end{thm}

One of the implications of the theorem is tautological: if $M$ is a closed $n$\nb-manifold with a $G$\nb-structure (i.e., an orientation if $G=\SO$ or a spin structure if $G=\spin$) and $M$ carries a \pscm, then the bordism class $[M,u]\in \Omega^G_n(B\pi)$ belongs to the subgroup $\Omega^{G,+}_n(B\pi)$, which  by definition consists of bordisms classes representable by pairs $(N,f\colon N\to B\pi)$, where $N$ carries a \pscm. 

The non-trivial statement is that $[M,u]\in \Omega_n^{G,+}(B\pi)$ implies that $M$ carries a \pscm. In other words, if $(M,u)$ is {\em $G$\nb-bordant} to a pair $(N,f)$ where $N$ carries a \pscm, then $M$ itself carries a \pscm. Unwinding the assumption that $(M,u)$ is $G$\nb-bordant to $(N,f)$, it means that there is a pair $(W,F)$, where 
\begin{itemize}[topsep=2pt,itemsep=-2pt]
\item $W$ is a $G$\nb-bordism between $M$ and $N$, i.e., a manifold of dimension $n+1$ equipped with a $G$\nb-structure whose boundary $\p W$ is the disjoint union of the $G$\nb-manifolds $M$ and $-N$, where $-N$ is the manifold $N$, equipped with the opposite orientation/spin structure. 
\item $F\colon W\to B\pi$ is a map which makes the following diagram commutative:
\begin{equation}\label{eq:bordism}
\begin{tikzcd}
M\ar[r,hook,"{i^M}"]\ar[rd,"u"']
&W\ar[d,"F"]
&\ar[l,hook',"{i^N}"'] N\ar[ld,"f"]\\
&B\pi&
\end{tikzcd}
\end{equation}
\end{itemize}
The idea for constructing a \pscm\ on $M$ is to show that $M$ can be constructed by a sequence of surgeries starting from $N$, and to use the Surgery Theorem \ref{thm:surgery} to propagate the given \pscm\ on $N$ to one on $M$. Morse theory shows that $M$ is obtained by a sequence of surgeries from $N$: pick a Morse function $g\colon W\to [0,1]$ on $W$ with $g_{|M}\equiv 0$ and $g_{|N}\equiv 1$ and consider the topology of the level sets $W_t=g^{-1}(t)$ for $t\in [0,1]$:
\begin{itemize}[topsep=2pt,itemsep=-2pt]
\item  If there is no critical value of $g$ in some interval $[t,t']$, then $W_{t}$ is diffeomorphic to $W_{t'}$.
\item If there is one critical point $x$ of index $i$ with $g(x)\in (t,t')$, then $W_{t}$ is obtained from $W_{t'}$ by a surgery of codimension $i$.
\end{itemize}
In particular, if all critical points have different values, which can be arranged for, then $M=W_0$ can be obtained from $N=W_1$ by a sequence of surgeries. Moreover, all these surgeries have codimension $\ge 3$ as required by the Surgery Theorey \ref{thm:surgery}, provided the Morse function $g$  has {\em no critical points of index $\le 2$}. 

A Morse function $g$ allows to calculate the relative homology groups $H_*(W,M)$ (with integer coefficients)  via the Morse chain complex associated to $g$. That chain complex has one copy of $\Z$ in degree $i$ for any critical point of index $i$, and hence the vanishing of the groups $H_k(W,M)$ for $0\le k\le 2$ is a necessary condition for the existence of a Morse function $g$ without critical points of index $0\le i\le 2$. A sufficient condition is that the inclusion map $i^M\colon M\into W$ is a $2$\nb-equivalence, i.e., the induced map on homotopy  $i^M_*\colon \pi_k(M)\to \pi_k(W)$ is an isomorphism for $k=0,1$ and is surjective for $k=2$. 

There is no reason that for the given bordism $W$ the inclusion map $i^M\colon M\into W$ is a $2$\nb-equivalence, but we claim that by surgeries in the interior of $W$ (i.e., without affecting the boundary $\p W=M\amalg N$), we can modify the bordism $(W,F\colon W\to B\pi)$ such that 
\begin{enumerate}[label=\normalfont(\roman*),topsep=2pt,itemsep=-2pt]
\item $F\colon W\ra B\pi$ is a $3$\nb-equivalence if $G=\spin$, and
\item $F\times c^{TW}\colon W\ra B\pi\times B\SO$ is a $3$\nb-equivalence if $G=\SO$. \end{enumerate}
Here $c^{TW}$ is a classifying map of stable tangent bundle of the oriented manifold $W$. 

Assuming the claim above, let us argue that the inclusion map $i^M\colon M\into W$ to the (modified) bordism $W$ is a $2$\nb-equivalence. The classifying map $u\colon M\to B\pi$ of the universal covering of $\wt M$ is always a $2$\nb-equivalence. Hence the fact that $F$ is a $3$\nb-equivalence in the case $G=\spin$ implies that $i^M$ is a $2$\nb-equivalence by the commutative diagram \eqref{eq:bordism}. In the case $G=\SO$, we instead look the commutative (up to homotopy) diagram
\[
\begin{tikzcd}
M\ar[r,hook,"{i^M}"]\ar[rd,"u\times c^{TM}"']
&W\ar[d,"F\times c^{TW}"]\\
&B\pi\times B\SO&
\end{tikzcd}
\]
The condition that the universal covering $\wt M$ of $M$ does not admit a spin structure implies that the classifying map $c^{TM}\colon M\to B\SO$ of the stable tangent bundle of $M$ induces a surjection $c^{TM}_*\colon \pi_2(M)\to \pi_2(B\SO)=\Z/2$. Hence the map $u\times c^{TM}$ is a $2$\nb-equivalence, and since $F\times c^{TW}$ is a $3$\nb-equivalence, it follows that $i^M$ is a $2$\nb-equivalence. 

To arrange for conditions (i) resp.\ (ii) by surgery on $W$, we can do $0$\nb-surgeries to make $W$ connected which ensures that the map $F_*\colon \pi_k(W)\to \pi_k(B\pi)$ is an isomorphism for $k=0$. For $k=1$, the map $F_*$ is surjective (since $u_*\colon\pi_1(M)\to \pi_1(B\pi)$ is an isomorphism). Modifying $W$ by surgeries of embedded circles which generate the kernel of $F_*$ ensures that $F_*$ is an isomorphism for $k=1$. Similarly, if $W$ is spin, the elements of $\pi_2(W)$  can be represented by embedded $2$\nb-spheres with trivial normal bundle, allowing to do surgery on these spheres to achieve $\pi_2(W)=0$. For $G=\SO$, only the elements in the kernel of $c^{TW}_*\colon \pi_2(W)\to \pi_2(B\SO)$ can be represented by embedded $2$\nb-spheres with trivial normal bundle; doing surgery on those makes the kernel of $c_*^{TW}$ trivial, and hence $F\times c^{TW}$ is a $3$\nb-equivalence.

\subsection{Simply connected manifolds with \pscm s}\label{ssec:sconn}
The goal of this section is to outline the proofs of Corollary \ref{cor:GL} and Theorem \ref{thm:stolz} which characterize those  simply connected closed manifolds of dimension $n\ge 5$ which carry \pscm s (these are restated as Theorem \ref{thm:scSO} resp.\ Theorem \ref{thm:stolz2} below). 
The Bordism Theorem  \ref{thm:bordism2} shows this amounts to determining the subgroup $\Omega^{G,+}_n\subset  \Omega^{G}_n$, $G=\SO,\Spin$ of the oriented (resp.\ spin) bordism group represented by manifolds with \pscm s. 

Based on prior work of Thom, Milnor and Dold, the calculation of the oriented bordism groups $\Omega^\SO_n$ was completed by Wall \cite{Wa1960}. In fact, he determined the structure of the $\Z$\nb-graded bordism ring
\[
\Omega_*^\SO:=\bigoplus_{n=0}^\infty\Omega_n^\SO,
\]
whose multiplication is given by the cartesian product of manifolds, and provided a list of   explicit closed oriented manifolds whose bordism classes multiplicatively generate this ring.  Gromov and Lawson noticed that {\em all} of these manifolds admit \pscm s, since each one of them can be identified with the total space of a fiber bundle $E\to N$ whose fibers are complex projective spaces, and whose structure group is the isometry group of the standard \pscm\ on the complex projective space (see Observation \ref{ob:fib}). By the Bordism Theorem this implies the following result.  

\begin{thm}\label{thm:scSO} {\bf (Gromov-Lawson, \cite[Cor.\ C]{GL1980}).}
Every closed simply-connected $n$-manifold, $n\ge 5$, which is not spin, carries a metric of positive scalar curvature.
\end{thm}

By contrast, the Lichnerowicz-Hitchin Theorem \ref{thm:LH} shows that the vanishing of the index invariant $\alpha(M)\in \KO_n$ is a necessary condition for a spin $n$\nb-manifold $M$ to carry a \pscm. In \cite{GL1980} Gromov and Lawson conjectured the following result, which was proved by the author \cite[Thm.\ A]{St1992}.

\begin{thm}\label{thm:stolz2} Let $M$ be a simply-connected spin manifold of dimension $n\ge 5$. Then $M$ carries a \pscm\ \Iff $\alpha(M)=0$.
\end{thm}

By the Lichnerowicz-Hitchin Theorem \ref{thm:LH}, there is an inclusion 
\begin{equation}\label{eq:spin_+}
\Omega^{\spin,+}_n\into \ker(\alpha\colon \Omega_n^\spin\ra \KO_n).
\end{equation}
Here  $\Omega_n^{\spin,+}\subset \Omega^\spin_n$ is the subgroup of bordism classes represented by manifolds carrying \pscm s, and $\alpha$ is the well-defined homomorphism given by sending $[M]\in \Omega^\spin_n$ to the index invariant $\alpha(M)\in \KO_n$. 
By the Bordism Theorem  \ref{thm:bordism2}(2), the theorem above is equivalent to showing that the inclusion \eqref{eq:spin_+} is an equality. Gromov and Lawson showed that the inclusion map is an isomorphism rationally \cite[Cor.\ B]{GL1980}, and  Miyazaki proved it after tensoring with $\Z[\frac 12]$ \cite{Miy1985}. Rosenberg showed equality for $n\le 23$ by exhibiting explicit generators for the kernel of $\alpha$ in that range \cite[Thm.\ 1.1]{Ro1986III}.  

The spin bordism groups $\Omega^\spin_n$ have been completely determined by Anderson, Brown and Peterson \cite{ABP1967}, but unlike for the oriented bordism groups, still today {\em no explicit manifolds} are known that generate them. The proof of  Theorem \ref{thm:stolz2} is based on Observation \ref{ob:fib}(2) that the total space of a fiber bundle $E\to N$ carries a \pscm, provided the fiber comes equipped with a \pscm, and the structure group is the isometry group of the fiber. A good candidate for the fiber is the quaternionic projective plane $\HP^2$, since its spin bordism class $[\HP^2]\in \Omega_8^\spin$ is a generator of the kernel of $\alpha$ in degree $8$, which is the first degree in which the kernel of $ \alpha$ is non-trivial. The standard Riemannian metric on $\HP^2$ has positive scalar curvature; its isometry group is the projective-symplectic group $\PSp(3)$. Hence Theorem \ref{thm:stolz2} is a consequence of the following completely topological result. 

\begin{thm}\label{thm:transfer} Every element in the kernel of $\alpha\colon \Omega^\spin_n\to \KO_n$ is represented by a total space of a fiber bundle with fiber $\HP^2$ and structure group $H=\PSp(3)$. 
\end{thm}

Localized at the prime $2$, this was proved by the author \cite[Thm.\ B]{St1992}, and by Kreck and the author in \cite[Prop.\ 4.2]{KrSt1993} after inverting $2$. We explain below what this means, and outline the proof of this result in the rest of this section.

Every fiber bundle $E\to N$ with fiber $\HP^2$ and structure group $H=\PSp(3)$ is the pull-back of the {\em universal  $\HP^2$\nb-bundle} $EH\times_H\HP^2\to BH$ via some map $f\colon N\to BH$. It turns out that a spin structure on $N$ induces a spin structure on the total space $E=f^*(EH\times_H\HP^2)$, and hence we can define a homomorphism 
\[
\Psi\colon \Omega_{n-8}^\spin(BH)\ra \Omega^\spin_n
\]
by mapping the bordism class of $f\colon N\to BH$ to the bordism class of the total space of the pullback bundle $f^*(EH\times_H\HP^2)$. 

We note that the composition 
\begin{equation}\label{eq:exact}
\begin{tikzcd}
\Omega_{n-8}^\spin(BH)\ar[rr,"\Psi"]
&&\Omega_n^\spin\ar[rr,"\alpha"]
&&\KO_n
\end{tikzcd}
\end{equation}
is trivial, since the image of $\Psi$ is represented by total spaces of $\HP^2$\nb-bundles; these carry \pscm s and hence their $\alpha$\nb-invariant is trivial by the Lichnerowicz-Hitchin Theorem \ref{thm:LH}. Theorem \ref{thm:transfer} is equivalent to the statement $\ker\alpha=\im\Psi$, i.e., to the exactness of the sequence \eqref{eq:exact} at the middle term. 

To prove this, it suffices to show exactness {\em localized at the prime $2$}, i.e., after tensoring the sequence with $\Z_{(2)}$, and {\em with $2$ inverted}, i.e., after tensoring with $\Z[\tfrac 12]$. Here
\begin{equation}\label{eq:loc_inv}
\Z_{(2)}=\left\{\tfrac ab\in\Q\mid \text{$b$ prime to $2$}\right\}
\qquad\text{and}\qquad
\Z\left[\tfrac 12\right]=\left\{\tfrac ab\in\Q\mid \text{$b$ is a power of $2$}\right\}.
\end{equation}

The $\alpha$\nb-invariant is multiplicative, i.e.,   $\alpha(M\times N)=\alpha(M)\alpha(N)$ 
for closed spin manifolds $M$, $N$. Here the product $\alpha(M)\alpha(N)\in \KO_*$ is given by the tensor product of the Clifford modules representing these elements (see \eqref{eq:KO_Cliff}). Put another way, the $\alpha$\nb-invariant gives a homomorphism of graded rings
\[
 \Omega_*^\spin:=\bigoplus_{n\ge 0}\Omega_n^\spin
\ \overset{\alpha}\ra\ \ko_*:=\bigoplus_{n\ge 0}\KO_n.
\]
Explicitly, $\ko_*=\Z[\eta,\omega,\mu]/(2\eta,\eta^3,\omega^2-4\mu)$, where $\eta,\omega,\mu$ are elements of degree $1,4$, and $8$ respectively. In fact,
\[
\eta=\alpha(S^1)
\qquad
\omega=\alpha(K)
\qquad
\mu=\alpha(B),
\]
where $S^1$ is the circle with the non-bounding spin structure, $K$ is the Kummer surface,  a degree $4$ hypersurface in $\CP^3$ (see Example \ref{ex:hyp}), and $B$ is any closed spin $8$\nb-manifold with $\wh A(K)=1$. 

\begin{proof}[Proof of Theorem \ref{thm:transfer} with $2$ inverted]
The spin bordism ring with $2$ inverted, i.e., the ring $\Omega_*^\spin\otimes \Z[\frac 12]$ is the polynomial ring $\Z[\frac 12][x_4,x_8,\dots]$ with generators $x_{4k}$ of degree $4k$. The generator $x_4$ can be chosen to be the bordism class $[K]$ of the Kummer surface. By a characteristic class calculation, it was shown in \cite[Prop.\ 4.2]{KrSt1993} that for $i\ge 2$ there are manifolds $M^{4i}$ of dimension $4i$ which are 
$\HP^2$\nb-bundles over a closed spin manifolds with structure group $\Sp(3)$ such
that the bordism classes
\begin{equation}\label{eq:family}
[K^4],\ [M^8],\ [M^{12}],\dots
\qquad\text{are generators of the polynomial ring $\Omega^\spin_*\otimes\Z[\frac 12]$}.
\end{equation}
The classes $[M^{4i}]$ are in the kernel of $\alpha$, since these manifolds carry \pscm s, while $\alpha(K)=\omega$  is the generator of the ring $\ko_*\otimes \Z[\frac 12]=\Z[\frac 12][\omega]$. This shows that the kernel of $\alpha$ with $2$ inverted is the ideal generated by the elements $[M^{4i}]$, for $i=2,3,\dots$. As $\HP^2$\nb-bundles, these generators are in the image of $\Psi$, consisting of all bordism classes represented by total spaces of $\HP^2$\nb-bundles. Since the image of $\Psi$ is an ideal, it contains all of $\ker \alpha$, thus proving Theorem \ref{thm:transfer} with $2$ inverted. 
\end{proof}

\begin{proof}[Proof of Theorem \ref{thm:transfer} localized at $2$]
The proof at the prime $2$ is more involved, due to the intricate structure of spin bordism at the prime $2$. Like the computation of the bordism groups $\Omega^\SO_n$ and $\Omega^\spin_n$, the proof of this statement is based on the Pontryagin-Thom isomorphism which expresses bordism groups as homotopy groups of associated Thom spectra. More precisely, the groups   in the sequence \eqref{eq:exact} can be expressed as homotopy groups of spectra by means  of a commutative diagram 
\begin{equation}\label{eq:hmtp1}
\begin{tikzcd}
\Omega_{n-8}^\spin(BH)\ar[rr,"\Psi"]\ar[d,"\cong"]
&&\Omega_n^\spin\ar[rr,"\alpha"]\ar[d,"\cong"]
&&\KO_n\ar[d,"\cong"]\\
\pi_n(\MSpin\wedge\Sigma^8BH_+)\ar[rr,"{T_*}"]
&&\pi_n(\MSpin)\ar[rr,"{U^\spin_*}"]
&&\pi_n(\ko)\\
\end{tikzcd}
\end{equation}
Here 
\begin{itemize}
\item  $MSpin$ is the Thom spectrum whose $n$\nb-th space $\MSpin_n$ is the Thom space of the vector bundle over $\BSpin(n)$ given by the pullback  of the universal $n$\nb-dimensional vector bundle $\gamma^n\to BO(n)$. The middle isomorphism is given by the Pontryagin-Thom construction.
\item The left isomorphism is also given by the Pontryagin-Thom isomorphism 
\[
\Omega^\spin_{n-8}(BH)\cong \pi_{n-8}(\MSpin\wedge BH_+)
\]
 composed with the suspension isomorphism 
 \[
 \pi_{n-8}(\MSpin\wedge BH_+)\cong \pi_n(\MSpin\wedge\Sigma^8BH_+).
 \]
\item The map $T\colon \MSpin\wedge\Sigma^8BH_+\to \MSpin$ is a {\em transfer map} associated to the $\HP^2$\nb-bundle $EH\times_H\HP^2\to BH$ \cite[section 2]{St1992}. 
\item $\ko$ is the connective version of the real $K$\nb-theory spectrum, i.e., $\pi_n(\ko)\cong \KO_n$ for $n\ge 0$ and $\pi_n(\ko)=0$ for $n<0$. The spectrum map $U^\spin\colon \MSpin\to \ko$ is the homotopy theoretic incarnation of Atiyah's real $K$\nb-theory orientation class of spin vector bundles. 
\end{itemize}

The composition $U^\spin\circ T$ induces the trivial map on homotopy groups by the diagram above, and the vanishing of $\alpha\circ \Psi$ by the Lichnerowicz-Hitchin Theorem \ref{thm:LH}. In fact, the composition itself is zero homotopic as an argument based on the Family Index Theorem shows \cite[Prop.\ 1.1]{St1992}. This implies that up to homotopy $T$ factors through a map $\wh T$ from $\MSpin\wedge\Sigma^8BH_+$ to the homotopy fiber $\wh\MSpin$ of $U^\spin$. This yields the commutative diagram whose bottom row is exact.
\begin{equation}\label{eq:hmtp2}
\begin{tikzcd}
\pi_n(\MSpin\wedge\Sigma^8BH_+)\ar[rr,"{T_*}"]\ar[d,"{\wh T_*}"]
&&\pi_n(\MSpin)\ar[rr,"{U^\spin_*}"]\ar[d,equal]
&&\pi_n(\ko)\ar[d,equal]\\
\pi_n(\wh \MSpin)\ar[rr]
&&\pi_n(\MSpin)\ar[rr,"{U^\spin_*}"]
&&\pi_n(\ko)
\end{tikzcd}
\end{equation}
It follows that exactness at the middle term of the top row is equivalent to surjectivity of $\wh T_*$. 

To prove surjectivity of the map $\wh T_*$ localized at $2$, use is made of the mod $2$ Adams spectral sequence which for a spectrum $X$ converges to $\pi_*X\otimes\Z_{(2)}$, the homotopy groups of $X$ localized at $2$. Its $E_2$\nb-page is given by Ext-groups built from the homology $H_*(X;\Z/2)$ viewed as comodule over the dual Steenrod algebra $A_*$. The map $\wh T$ induces a map of Adams spectral sequences, which on the $E_2$\nb-page is determined by the map induced by $\wh T$ on $\Z/2$\nb-homology. A calculation shows that this homology map is a split surjection of $A_*$\nb-comodules \cite[Prop.\ 1.3]{St1992}, and hence the map induced by $\wh T$ on the $E_2$\nb-page is surjective. This does {\em not} imply that the map induced by $\wh T$ on the $E_\infty$\nb-page is surjective, since there could be non-trivial differentials in the domain spectral sequence. Fortunately, this is not the case \cite[Prop.\ 1.5]{St1992}, and so $\wh T$ does induce a surjection on the $E_\infty$\nb-page. This implies that the map induced by $\wh T$ is surjective on $2$\nb-local homotopy groups, which proves Theorem \ref{thm:transfer} localized at $2$. 
\end{proof}

\subsection{Reduction to homology and $\ko$-homology}
\label{ssec:H_ko}

This section is an outline of the proof of Theorem \ref{thm:ko}, according to which a closed connected $n$\nb-manifold $M$, $n\ge 5$, with fundamental group $\pi$ carries a \pscm\ \Iff $u_*[M]^\spin\in \ko_n^+(B\pi)$ (if $M$ is spin) resp.\ $u_*[M]^\SO\in H_n^+(B\pi)$ (if $M$ is oriented and the universal covering $\wt M$ is non-spin). The proof is based on the Bordism Theorem \ref{thm:bordism2}, whose statement is completely analogous, but the condition for the existence of a \pscm\ is $[M,u]\in \Omega_n^{\spin,+}(B\pi)$ resp.\ $[M,u]\in \Omega_n^{\SO,+}(B\pi)$. These conditions are related by natural transformations of homology theories
\begin{equation}\label{eq:nat_trans}
U^\spin_*\colon \Omega_n^\spin(X)\ra \ko_n(X)
\qquad\qquad
U^\SO_*\colon \Omega_n^\SO(X)\ra H_n(X)
\end{equation}
given by sending a bordism class $[f\colon M\to X]\in \Omega^G_n(X)$ to $f_*[M]^G$ for $G=\spin,\SO$. By definition, $\ko_n^+(X)=U^\spin_*(\Omega_n^{\spin,+}(X))$ and $H_n^+(X)=U^\SO_*(\Omega_n^{\SO,+}(X))$, and hence to prove Theorem \ref{thm:ko}, it suffices to show
\begin{equation}\label{eq:ko_H}
\ker U^G_*\subset \Omega_*^{G,+}(X)
\qquad\text{for $G=\spin,\SO$}.
\end{equation}

The first step towards proving this is to show that the map  $U_*^G$ is part of an exact sequence, allowing us to think of the kernel $\ker U^G_*$ as the image of a map. This is achieved in an abstract way, by looking at the spectra representing these homology theories, as well as the spectrum maps inducing the natural transformations \eqref{eq:nat_trans}, and taking the homotopy fibers of these maps.

We recall that a spectrum $E$ determines  a (generalized) homology theory $E_*(X)$ by defining the {\em $n$-th $E$\nb-homology group} $E_n(X)$ of a topological space $X$ by $E_n(X):=\pi_n(E\wedge X_+)$. Conversely, every homology theory comes from a spectrum. For example, the Pontryagin-Thom construction yields isomorphisms
\begin{align*}
\Omega^G_n(X)&\cong \pi_n(\MG\wedge X_+)=\MG_n(X)
\qquad\text{for $G=\spin,\SO$}.
\end{align*}
This shows that homology theories $\Omega_n^\spin(X)$ (resp.\ $\Omega_n^\SO(X)$) are the homology theories associated to the Thom spectra $\MSpin$ resp.\ $\MSO$. Integral homology $H_n(X)$ is associated to the {\em integral Eilenberg Mac Lane spectrum $\HZ$} (which up to homotopy equivalence is determined by $\pi_0(\HZ)\cong \Z$ and $\pi_n(\HZ)=0$ for $n\ne 0$). The connective real $K$\nb-homology $\ko_n(X)$ is by definition the homology theory associated to the real connective $K$\nb-theory spectrum $\ko$. 

A map between spectra $U\colon E\to F$ determines a natural transformation 
$
 E_n(X)\overset{U_*}\to F_n(X)
$
given by 
\[
\begin{tikzcd}
E_n(X)=\pi_n(E\wedge X_+)\ar[rr,"{(U\wedge \id_X)_*}"]
&&\pi_n(F\wedge X_+)=F_n(X)
\end{tikzcd}.
\]
 The natural transformations \eqref{eq:nat_trans} are given by maps of spectra
$U^\spin\colon \MSpin\ra \ko$ resp.\ 
$U^\SO\colon \MSO\ra \HZ$. By taking the homotopy fibers of $U^\spin$ resp.\ $U^\SO$, we obtain homotopy fibrations
\[
\begin{tikzcd}
\wh {\MSpin}\ar[r,"{i^\spin}"]
&\MSpin\ar[r,"{U^\spin}"]
&\ko
\end{tikzcd}
\qquad\text{and}\qquad
\begin{tikzcd}
\wh {\MSO}\ar[r,"{i^\SO}"]
&\MSO\ar[r,"{U^\SO}"]
&\HZ
\end{tikzcd},
\]
and the associated long exact sequences of homology groups (coming from the long exact sequences of homotopy groups  of the fibrations obtained by smashing the  fibrations above with $X_+$). In particular, the kernel of $U^G_*$ is the image of $i^G_*$ for $G=\spin,\SO$, and hence \eqref{eq:nat_trans} is equivalent to
\begin{align} \label{eq:i^G}
&\text{image}\left(i^G_*\colon \wh\MG_n(X)\to \MG_n(X)\right)
\ \subseteq\ \Omega^{G,+}_n(X)
\qquad\text{for $G=\spin,\SO$}.
\end{align}
It suffices to prove  these containment relations localized at $2$  and with $2$ inverted (see \eqref{eq:loc_inv}). 
The method of proof for those two cases is quite different: 
\begin{itemize}
\item Localized at $2$, the proof is based on a stable homotopy theoretic understanding of the spectra $\wh\MSpin$, $\wh \MSO$, and relating the image of $i^\spin_*$ to $\HP^2$\nb-bundles.
\item With $2$ inverted, it  is based on a geometric interpretation of the groups $\wh \MG_*(X)$ as bordism groups of manifolds with additional structure. 
\end{itemize}

\begin{proof}[Proof of \eqref{eq:i^G} for $G=\spin$ localized at $2$] We recall from the proof of Theorem \ref{thm:transfer} that the transfer map $T$ factors in the form
\[
\begin{tikzcd}
&&\wh\MSpin\ar[d,"{i^\spin}"]\\
\MSpin\wedge\Sigma^8BH_+\ar[rr,"T"']\ar[rru,"{\wh T}"]
&&\MSpin,
\end{tikzcd}
\]
and that the key for the proof of Theorem \ref{thm:transfer} is the fact that the induced map $\wh T_*$ on homotopy groups is surjective (localized at $2$). The key fact needed for the proof here is the stronger statement that $\wh T$ is a {\em split surjection of spectra} localized at $2$ \cite[Prop.\ 8.3]{St1994}, the  main technical result of that paper.  It implies that the induced map
\[
\wh T_*\colon (\MSpin\wedge \Sigma^8BH_+)_n(X)\ra \wh \MSpin_n(X)
\]
is surjective (localized at $2$) for any space $X$. Hence the image of $i_*^\spin$ is contained in the image of $T_*\colon (\MSpin\wedge \Sigma^8BH_+)_n(X)\to \MSpin_n(X)=\Omega_n^\spin(X)$. Geometrically, the image of $T_*$ is represented by total spaces of $\HP^2$\nb-bundles. These admit \pscm s, and hence the image of $i_*^\spin$ is contained in $\Omega_n^{\spin,+}(X)$. 
\end{proof}

\begin{proof}[Proof of \eqref{eq:i^G} for $G=\SO$ localized at $2$] This case is much simpler. For the convenience of the reader we repeat the argument of \cite[proof of Thm.\ 4.11]{RS2001}). Localized at $2$ the spectra $\MSO$ and $\wh\MSO$ are Eilenberg Mac Lane spectra, which implies that localized at $2$,
\[
\MSO_n(X)\cong \bigoplus_{j\ge 0}H_{n-j}(X;\Omega_j^\SO)
\qquad\text{and}\qquad
\wh\MSO_n(X)\cong \bigoplus_{j> 0}H_{n-j}(X;\Omega_j^\SO).
\]
The summand $H_{n-j}(X;\Omega_j^\SO)\subset \Omega_n^\SO(X)$ is given by bordism classes of the form 
\[
[f\colon M^{n-j}\times P^j\to X],
\]
 where $M^{n-j}$, $P^j$ are oriented closed manifolds of the indicated dimension, and $f$ factors through $M$. Since every bordism class in $\Omega_j^\SO$ for $j>0$ can be represented by a manifold that carries a \pscm, (this is the key fact for the proof of Gromov-Lawson Theorem \ref{thm:scSO}), this shows that the image of $i_*^{\SO}$ is contained in $\Omega_n^{\SO,+}(X)$. 
\end{proof}

\begin{proof}[Proof of \eqref{eq:i^G} with $2$ inverted] F\"uhring has given 
a geometric interpretation of the homology groups $\wh \MSpin_n(X)$ and $\wh \MSO_n(X)$ with $2$ inverted in terms of  bordism classes of $n$\nb-dimensional closed $\sP$\nb-manifolds equipped with a map to $X$ \cite[]{Fu2013}. Here $\sP=\{P_1,P_2,\dots\}$ is a family of smooth closed manifolds, and a $\sP$\nb-manifold $M$ of dimension $n$ is a smooth $n$\nb-manifold equipped with an additional structure \cite[Def.\ 2.1]{Fu2013}. Very roughly, it consists of a decomposition $M=A_1\cup\dots\cup A_k$ of $M$ into $n$\nb-dimensional submanifolds $A_i$ of the form $P_i\times B_i$. Over two-fold intersections $A_i\cap A_j$, these product decompositions are compatible in the sense that $A_i\cap A_j$ has the form $P_i\times P_j\times B_{ij}$ refining the product decompositions of $A_i$ and $A_j$ over the intersection $A_i\cap A_j$, and so forth for higher intersections. 

We remark that given a list of manifolds $\sP$, there is the Baas-Sullivan theory which considers bordism groups of manifolds with singularities which are built inductively from manifolds on the list $\sP$. Removing a small open neighborhood of the singularity from such a manifold with Baas-Sullivan type singularity, leads to a smooth manifold whose boundary has the structure described above. 

For any collection $\sP$, the bordism classes of $\sP$\nb-manifolds with maps to a topological space $X$ give a homology theory $\sP_*(X)$ \cite[Prop.\ 2.11]{Fu2013}. Moreover, the homology theories $\wh\MSpin_*(X)$ and $\wh\MSO_*(X)$ are given by specific collections $\sP^\spin$, $\sP^\SO$ of spin (resp.\ oriented) manifolds \cite[Prop.\ 2.12 and Prop.\ 2.11]{Fu2013}:
\[
 \wh \MSpin_*(X)\cong \sP^\spin_*(X)
 \qquad
 \wh \MSO_*(X)\cong \sP^\SO_*(X).
 \]
 The map 
$
 i^G_*\colon \wh \MG_*(X)\cong \sP^G_*(X)\ra \MG_*(X)=\Omega_*^G(X)
$
 for $G=\SO,\spin$ has the simple geometric interpretation of forgetting the additional structure on the smooth spin/oriented $\sP^G$\nb-manifolds representing the elements of $\sP^G_*(X)$. Moreover, a smooth closed manifolds with a $\sP^G$\nb-structure always carries a \pscm, provided all manifolds on the list $\sP^G$ do \cite[Thm.\ 3.1]{Fu2013}. This shows that the image of $i_*^G$ is contained in $\Omega_*^{G,+}(X)$ as claimed, provided the collection of manifolds $\sP^G$ can be chosen to consist of manifolds which carry \pscm s. 
 
 \medskip
 
We recall from \eqref{eq:family} that 
\[
\Omega_*^\spin\otimes\Z[\tfrac 12]=\Z[\tfrac 12][[K],[M^8],[M^{12}],\dots],
\]
where $K$ is the Kummer surface and $M^{4i}$ is a spin $4i$\nb-manifold which is the total space of an $\HP^2$\nb-bundle and  hence carries a \pscm. F\"uhring shows that one can choose $\sP^\spin=\{M^8,M^{12},\dots\}$ \cite[Prop.\ 2.12]{Fu2013}. He furthermore shows that $\sP^\SO$ can be chosen to be any collection of oriented manifolds whose bordism classes are generators of the polynomial ring $\Omega_*^\SO\otimes\Z[\tfrac 12]$ \cite[Prop.\ 2.11]{Fu2013}. For example, we could  choose $\sP^\SO=\{K,M^8,M^{12},\dots\}$, since the natural map $\Omega^\spin_*\to \Omega^\SO_*$ is an isomorphism after inverting $2$, but this is not a good idea, since the Kummer surface $K$ does not carry a \pscm\ (due to $\wh A(K)\ne 0$). However, we can replace $K$ by the complex projective plane $\CP^2$, and $\sP^\SO=\{\CP^2,M^8,M^{12},\dots\}$ is a family of oriented manifolds which carry \pscm s, and whose bordism classes are polynomial generators of $\Omega_*^\SO\otimes\Z[\frac 12]$.
\end{proof}

\subsection{Positive scalar curvature metrics on manifolds with non-trivial fundamental groups}\label{ssec:nsconn}

In this section we illustrate how Theorem \ref{thm:ko} is used to prove the GLR Conjecture for some fundamental groups $\pi$ by outlining the proof for finite groups with periodic cohomology \cite{BGS1997}. Finally we discuss Schick's counter example to the GLR conjecture \cite{Sch1998}.

\medskip

\noindent{\bf The GLR Conjecture for groups  $\pi$ with periodic cohomology \cite{BGS1997}.} 
To prove the GLR Conjecture for a finite group $\pi$, it suffices by \cite[Prop.\ 1.5]{KwSch1990} to prove it for the $p$\nb-Sylow subgroups of $\pi$. If $\pi$ has periodic cohomology, these are either cyclic or quaternion (or rather, generalized quaternion groups,  whose order is some power of $2$). For odd order cyclic groups the GLR Conjecture was proved in \cite{Ro1986III} (for prime order) and for the general case in \cite[Thm.\ 1.8]{KwSch1990}. Here we outline the argument in \cite{BGS1997} that proves the GLR Conjecture for cyclic groups $C_\ell$ and quaternionic groups $Q_\ell$ of order $\ell=2^k$. 

To prove the GLR Conjecture for a group $\pi$, it suffices by Theorem \ref{thm:ko} to show that  the inclusion 
\[
\ko_n^+(B\pi)\ \subset\ \ker(\ko_n(B\pi)\overset{A\circ p}\ra \KO_n(B\pi))=:\ker_n(A\circ p)
\]
is an equality. 
Since the GLR Conjecture is true for the trivial group, it suffices to show $\wt\ker_n(B\pi)\subset \wt \ko_n^+(B\pi)$ (given by passing to the kernel of the map $\ko_n(B\pi)\to \ko_n$). The strategy used  in \cite{BGS1997} to prove this for $\pi=C_\ell$ and $\pi=Q_\ell$, $\ell=2^k$, is to consider lens spaces and lens space bundles over $S^2$ (in the case $\pi=C_\ell$) and lens spaces and quotients of free actions of $Q_\ell$ on spheres of dimension $\equiv 3\mod 4$ (for $\pi=Q_\ell$). Let $\cM_*(B\pi)\subset\wt\Omega_*^\spin(B\pi)$ be the $\Omega_*^\spin$\nb-submodule generated by these manifolds and their natural maps to $B\pi$, and consider its image under the natural transformation $U^\spin_*\colon \wt\Omega_*^\spin(B\pi)\to\wt\ko_*(B\pi)$. Since the lens spaces and the quaternionic space forms carry \pscm s, there is the following chain of inclusions:
\[
U_*^\spin(\cM_n(B\pi))\ \subseteq\ \wt \ko_n^+(B\pi)\ \subseteq\ \wt\ker_n(A\circ p).
\]
Proving the equality $U_*^\spin(\cM_n(B\pi))=\wt\ker_n(A\circ p)$ (this is Theorem 2.3 in \cite{BGS1997}) then implies the desired equality $\wt \ko_n^+(B\pi)=\wt\ker_n(A\circ p)$. Due to the multiplicative structure (given by multiplication by $\eta\in \ko_1$), it suffices to prove $U_*^\spin(\cM_n(B\pi))=\wt\ker_n(A\circ p)$ for $n\equiv 3\mod 4$ for $\pi=Q_\ell$, and for $n$ odd for $\pi=C_\ell$. For these $n$, it is shown that  the order of the finite group $U_*^\spin(\cM_n(B\pi))$ is greater or equal to the order of $\wt\ker_n(A\circ p)$. An upper bound for the order of $\wt\ker_n(A\circ p)$ is obtained by using the Atiyah-Hirzebruch spectral sequence converging to $\wt \ko_*(B\pi)$ to obtain an upper bound on $\wt \ko_n(B\pi)$, in conjunction with information about the assembly map $A\colon \KO_*(B\pi)\to \KO_*(\R\pi)$ for finite $2$\nb-groups $\pi$ (which is obtained by dualizing the Atiyah-Segal results concerning $\KO^*(B\pi)$). A lower bound for $U_*^\spin(\cM_n(B\pi))$ is obtained by calculating the eta-invariants of twisted Dirac operators on the lens spaces resp.\ quaternionic space forms generating $\cM_*(B\pi)$. The  eta-invariant associated to an irreducible representation of $\pi$ provide homomorphisms from $\wt {\KO}_n(B\pi)$ to $\R/\Z$ resp.\ $\R/2\Z$ (depending on whether the representation of complex, real or quaternionic type). Using all irreducible representations, one obtains a homomorphism $\check \eta$ from $\wt\ko_n(B\pi)$ to a sum of copies of $\R/\Z$ and $\R/2\Z$. An interesting byproduct of these calculations is that the homomorphism
\[
\begin{tikzcd}
\wt \ko_n(B\pi)\ar[r,"p"]
&\wt {\KO}_n(B\pi)\ar[r,"{\check \eta}"]
&\R/\Z\oplus\dots\oplus \R/\Z\oplus\R/2\Z\oplus\dots\oplus \R/2\Z
\end{tikzcd}
\]
is injective if $\pi=C_\ell$ and $n$ is odd, or if $\pi=Q_\ell$ and $n\equiv 3\mod 4$ \cite{BGS1997}.

\medskip

\noindent{\bf Schick's counter example to the GLR conjecture \cite{Sch1998}.}
Let $\rho\colon \Z^5\to \Z^4\times\Z/3=:\pi$ be the product of the identity map on the first four components, and the projection map $\Z\to \Z/3$ on the last. Then the map
$B\rho\colon T^5=B\Z^5\overset{B\rho}\ra B\pi$ represents an element of the spin bordism $\Omega_5^\spin(B\pi)$ (equip each $S^1$ factor of $T^5$ with the bounding spin structure). This is an element of order $3$, and hence its image in $\KO_5(C^*\pi)$ is trivial, since the torsion of $\KO_*(C^*\pi)$ is only $2$\nb-torsion \cite[Prop.\ 2.1]{Sch1998} (this is the case for the product of a free abelian group with {\em any} finite group).  
A $1$-surgery on $T^5$ on an embedded  $1$\nb-sphere representing the kernel of $B\rho_*\colon \pi_1(T^5)\to \pi_1(B\pi)$ produces a spin manifold $M$ with fundamental group $\pi$ such that $[M,u]=[T^5,B\rho]\in \Omega_5^\Spin(B\pi)$. Then the index obstruction $\alpha(M,u)\in \KO_5(C^*\pi)$ is the image of $[T^5,B\rho]$ in $\KO_5(C^*\pi)$ and hence zero. 

To show that $M$ does not carry a \pscm, we use Corollary \ref{cor:min_hyp} (which exploits the obstructions coming from stable minimal hypersurfaces). Let $\alpha_i\in H^1(B\pi)$ be pullback of the generator of $H^1(S^1)$ via the projection of $B\pi\to S^1$ to the $i$\nb-th factor, $i=1,\dots,4$. Assume that $M$ does carry a \pscm, and hence $u_*[M]^\SO=B\rho_*[T^5]$ belongs to the positive subgroup $H_5^+(B\pi)\subseteq H_5(B\pi)$. Applying  Cor.\ \ref{cor:min_hyp} three times, then  $\alpha_1\cap (\alpha_2\cap (\alpha_3\cap B\rho_*[T^5]))$ belongs to $H_2^+(B\rho)$. This is the desired contradiction, since this iterated cap product is non-trivial, but by the Gauss-Bonnet Theorem the only closed $2$\nb-manifold that carries a \pscm\ is the $2$\nb-sphere, and any homology class in $B\pi$ represented by a map $f\colon S^2\to B\pi$ is trivial.

\section{Some open questions}
\begin{enumerate}[leftmargin=0cm]
\item The focus of this survey is the question \ref{quest:main} which closed manifolds $M$ carry \pscm s. If $M$ is a manifold with boundary $\p M$, and $h$ is a \pscm\ on $\p M$, one can ask the corresponding ``relative'' question: does $h$ extend to a \pscm\ $g$ on $M$? We require that $g$ is a {\em product metric near the boundary}, a tacit assumption we will make for metrics on manifolds with boundary. If $M$ is a spin $n$\nb-manifold, the Clifford linear Dirac operator $D^M_{\Cl_n}$ on $M$ with respect to the Atiyah-Singer-Patodi boundary conditions has a Clifford index $\alpha(M,h)\in \KO_n$. The operator $D^M_{\Cl_n}$ is constructed using {\em any} metric $g$ on $M$ extending $h$. The Clifford index $\alpha(M,h)$ is independent of the choice of $g$, and depends only on the concordance class of the \pscm\ $h$ on $\p M$ (two \pscm s $h,h'$ on a closed manifold $N$ are {\em concordant} if there is a \pscm\ on $N\times [0,1]$ which restricts to $h$ resp.\ $h'$ on the boundary). Moreover, if the scalar curvature of $g$ is positive, then, as in the case of closed manifolds, the Dirac operator $D^M_{\Cl_n}$ is invertible, and hence $\alpha(M,h)=0$. This shows that the index invariant $\alpha(M,h)$ is an obstruction to extending $h$ to a \pscm\ $g$ on $M$. According to Theorem \ref{thm:stolz} this is the {\em only} obstruction if $M$ is a simply connected compact spin manifold of dimension $n\ge 5$ without boundary. So it is very natural to ask: 

\begin{question}  Let $M$ be a simply connected compact spin manifold of dimension $n\ge 5$ and $h$ a \pscm\ on $\p M$. Does $h$ extend to a \pscm\ on $M$ \Iff the index obstruction $\alpha(M,h)\in \KO_n$ vanishes?
\end{question}

As discussed for closed manifolds, the index invariant $\alpha(M,h)$ has a refinement which lives in $\KO_n(C^*\pi)$ where $\pi$ is the fundamental group of $M$. The answer to the analogous question is ``no'' in general, since already for closed manifolds, there are the additional obstructions coming from the stable minimal hypersurface method (see  \S  \ref{ssec:min_hyp}).

The above question is intimately related to the classification of  \pscm s up to concordance. Let $R_n$ be the bordism group of pairs $(M,h)$ consisting of a spin $n$\nb-manifold $M$ and a \pscm\ $h$ on $\p M$.

\begin{thm} Let $M$ be a smooth compact simply connected spin manifold of dimension $n\ge 5$.
\begin{enumerate}[label=\normalfont(\roman*),topsep=2pt,itemsep=-2pt]
\item {\bf Hajduk \cite{Haj1991}:} A \pscm\ $h$ on $\p M$ extends to a \pscm\ on $M$ \Iff $[M,h]\in R_n$ vanishes.
\item {\bf \cite[Thm.\ 3.9]{St1995}, \cite[Thm.\ 1.1]{St1996}:} If $h$ extends to a \pscm\ on $M$, then the group $R_{n+1}$ acts freely and transitively on the set of concordance classes of such metrics. 
\end{enumerate}
\end{thm}

This is just the simplest instance of a much more general result. Without assuming  simply connectivity for $M$, the bordism group $R_n$ has to be replaced by the bordism group $R_n(\pi)$, $\pi=\pi_1(M)$, where all manifolds come equipped with maps to the classifying space $B\pi$ (\cite[Thms.\ 3.8 \& 3.9]{St1995}); without the spin condition, it needs to be replaced by $R_n(\gamma(M))$ where $\gamma(M)$ is the {\em fermionic fundamental group} of $M$ (which depends only on the fundamental group  and the first two Stiefel-Whitney classes of $M$ up to isomorphism).

Associating to a pair $(M,h)$ the index obstruction $\alpha(M,h)\in \KO_n$ defines a homomorphism $\alpha\colon R_n\ra \KO_n$. It is easy to show that $\alpha$ is surjective, and a  positive answer to the question above would imply that $\alpha$ is an isomorphism; in particular, $\KO_{n+1}$ would act freely and transitively on the set of concordance classes of \pscm s on any compact manifold $M$ of  dimension $n\ge 5$ (and extending a given \pscm\ $h$ on $\p M$ if $\p M\ne \emptyset$), provided there exist such metrics. 

Alas, the question above appears out of reach of current methods. For example, consider a simply connected $5$\nb-manifold $M$. If the boundary is empty, the answer to the question whether $M$ carries a \pscm\ depends only on the bordism class of $M$ in $\Omega^\Spin_5$ (if $M$ is spin) or $\Omega^\SO_5$ (if $M$ does not admit a spin structure). Since both bordism groups are trivial, in either case $M$ carries a \pscm. 

If $\p M$ is non-empty, and $h$ is a \pscm\ on $\p M$, there is no known obstruction to extending $h$ to a \pscm on $M$, but it seems that some more direct geometric method that uses the metric $h$ in an essential way, is necessary. 

\item Let $M$ be a closed connected spin manifold $M$ of dimension $n\ge 5$ with fundamental group $\pi$. According to the Bordism Theorem \ref{thm:bordism} the answer to the question whether $M$ carries a \pscm\ depends only on the class it represents in the spin bordism group $\Omega^\spin_n(B\pi)$. By Theorem \ref{thm:ko} in fact it depends only on the image of this class under the map
\[
U_*^\spin\colon \Omega_n^\spin(B\pi)\ra \ko_n(B\pi).
\]
With the weaker hypothesis that the universal covering $\wt M$ is spin, the manifold $M$ represents an element in the {\em twisted spin bordism group} $\Omega_n^{\spin,\tau}(B\pi)$, where the twist $\tau$ depends on the first two Stiefel-Whitney classes of $M$ (the condition that $\wt M$ is spin guarantees that $w_2(M)\in H^2(M;\Z/2)$ comes from a unique class in $H^2(B\pi;\Z/2)$). There is a twisted version of connective real $K$\nb-theory, and a homomorphism 
\[
U_*^{\spin,\tau}\colon \Omega_n^{\spin,\tau}(B\pi)\ra \ko_n^\tau(B\pi),
\]
constructed in \cite{HJ2020} by a mixture of homotopy theoretic and operator theoretic methods. As in the non-twisted case, the hope is that it suffices to look at the associated element in $\ko_n^\tau(B\pi)$ to decide whether $M$ admits a \pscm. This boils down to a positive answer to the following quesiton.

\begin{question}  Is the kernel of $U_*^{\spin,\tau}$ representable by 
manifolds that carry \pscm s?
\end{question}

The paper \cite{HJ2020} lays the foundation for possible future homotopy theoretic arguments addressing this question. 

\item As mentioned in section \ref{ssec:min_hyp}, the only known obstructions to \pscm s on a closed manifold with non-spin universal covering comes from the stable minimal hypersurface method. This leads to the restrictions on the subgroup $H_n^+(X)\subset H_n(X)$ of  homology classes representable by closed oriented manifolds with \pscm s, expressed by Cor.\ \ref{cor:min_hyp}. 

\begin{question} Does the statement of Corollary \ref{cor:min_hyp} hold without the dimension restriction $n\le 7$?
\end{question}

The papers \cite{Lo2006} and \cite{SY2017} deal with ways to extend the stable minimal hypersurface method to higher dimension. The technical challenge is to develop techniques to deal with the singularities of stable minimal hypersurfaces.

\item Let $M$ be a closed connected manifold of dimension $n\ge 5$ with finite fundamental group $\pi$ whose universal covering $\wt M$ is non-spin. 
According to  Conjecture \ref{conj:nonspin_cover}, such a manifold carries a \pscm, and by the induction result \cite[Prop.\ 1.5]{KwSch1990}, it suffices to show this for finite $p$\nb-groups. As mentioned in the paragraph following Conjecture \ref{conj:nonspin_cover}, a lot is known if $\pi$ is an {\em abelian} $p$\nb-group. The essential open case is the 
the following for odd $p$.

\begin{question}  Let $\pi$ be the elementary abelian group $(\Z/p)^n$, let $\rho\colon \Z^n\to (\Z/p)^n$ be the projection map, and let  $B\rho\colon T^n=B\Z^n\to B\pi$ be the induced map of classifying spaces. Does $B\rho_*[T^n]^\SO\in H_n(B\pi)$ belong to $H_n^+(B\pi)$ for $n\ge 5$? (here  $T^n=S^1\times\dots\times S^1$ is the $n$\nb-dimensional torus, and $[T^n]^\SO\in H_n(T^n)$ is its fundamental class).
\end{question}

Joachim has shown that the answer for $p=2$ is affirmative \cite{Jo2004}, but his approach does not work for $p$ odd. 

The case  non-abelian finite $p$\nb-groups $\pi$ does not seem to have been studied much. The fact that the classifying space $B\pi$ can be ``built'' from the classifying spaces of its elementary abelian $p$\nb-groups  \cite[Thm.\ 1.4.\ and 1.12]{Dw1997}, suggests the following vague question.

\begin{question}  Can $H_*^+(B\pi)$ for a $p$\nb-group $\pi$ be determined in terms of $H_*^+$ for its $p$\nb-Sylow subgroups?
\end{question}

\end{enumerate}



\begin{thebibliography}{blah}

\bibitem[ABP1967]{ABP1967} Anderson, D. W.; Brown, E. H., Jr.; Peterson, F. P., \textit{The structure of the Spin cobordism ring}. Ann. of Math. (2) 86 (1967), 271–298

\bibitem[ABS1963]{ABS1963}  Atiyah, M. F.; Bott, R.; Shapiro, A, \textit{Clifford modules}, Topology 3 (1964), no. suppl, suppl. 1, 3--38.

\bibitem[BGS1997]{BGS1997}  Botvinnik, Boris; Gilkey, Peter; Stolz, Stephan, \textit{The Gromov-­Lawson­-Rosenberg conjecture for groups with periodic cohomology}, J. Differential Geom. 46 (1997), no. 3, 374 --
405.

\bibitem[BR2002]{BR2002} Botvinnik, Boris; Rosenberg, Jonathan \textit{The Yamabe invariant for non-simply connected manifolds}, J. Differential Geom. 62 (2002), no. 2, 175--208

\bibitem[BR2005]{BR2005} Botvinnik, Boris; Rosenberg, Jonathan,
\textit{Positive scalar curvature for manifolds with elementary abelian fundamental group}, Proc. Amer. Math. Soc. 133 (2005), no. 2, 545--556.

\bibitem[BR2021]{BR2021} Botvinnik, Boris; Rosenberg, Jonathan, \textit{Positive scalar curvature on Pin$^\pm$- and Spin$^c$-manifolds and manifolds with singularities}, in this volume

\bibitem[DL2013]{DL2013} Davis, James F.; Lück, Wolfgang \textit{The topological K-theory of certain crystallographic groups},  J. Noncommut. Geom. 7 (2013), no. 2, 373--431.

\bibitem[DP2003]{DP2003} Davis, James F. ; Pearson, Kimberly,
\textit{The Gromov-Lawson-Rosenberg conjecture for cocompact Fuchsian groups}, Proc. Amer. Math. Soc. 131 (2003), no. 11, 3571--3578


\bibitem[Dw1997]{Dw1997} Dwyer, W. G., \textit{Homology decompositions for classifying spaces of finite groups}, Topology 36 (1997), no. 4, 783--804

\bibitem[DSS2003]{DSS2003} Dwyer, William; Schick, Thomas; Stolz, Stephan
, \textit{Remarks on a conjecture of Gromov and Lawson}, High-dimensional manifold topology, 159--176, World Sci. Publ., River Edge, NJ, 2003.

\bibitem[Fu2013]{Fu2013} F\"uhring, Sven,
\textit{A smooth variation of Baas-Sullivan theory and positive scalar curvature},  Math. Z. 274 (2013), no. 3-4, 1029--1046

\bibitem[GL1980]{GL1980} Gromov, Mikhael; Lawson, H. Blaine, Jr. \textit{The classification of simply connected manifolds of positive scalar curvature}, Ann. of Math. (2) 111 (1980), no. 3, 423–434

\bibitem[GL1983]{GL1983} Gromov, Mikhael; Lawson, H. Blaine, Jr. \textit{Positive scalar curvature and the Dirac operator on complete Riemannian manifolds}, Inst. Hautes \'Etudes Sci. Publ. Math. No. 58 (1983), 83--196 (1984)

\bibitem[Haj1991]{Haj1991} Hajduk, Bogusław, \textit{On the obstruction group to existence of Riemannian metrics of positive scalar curvature}, Global differential geometry and global analysis (Berlin, 1990), 62--72, Lecture Notes in Math., 1481, Springer, Berlin, 1991.

\bibitem[Ha2016]{Ha2016} Hanke, Bernhard, \textit{Bordism of elementary abelian groups via inessential Brown-Peterson homology}, J. Topol. 9 (2016), no. 3, 725--746.

\bibitem[Ha2019]{Ha2019} Hanke, Bernhard, \textit{Positive scalar curvature on manifolds with odd oder abelian fundamental groups}, arXiv:1908.00944v3, to appear in Geometry \& Topology

\bibitem[HJ2020]{HJ2020} Hebestreit, Fabian; Joachim, Michael
\textit{Twisted spin cobordism and positive scalar curvature}, 
J. Topol. 13 (2020), no. 1, 1--58

\bibitem[Hi1974]{Hi1974}  Hitchin, Nigel, \textit{Harmonic spinors}, Advances in Math. 14 (1974), 1--55

\bibitem[Jo2004]{Jo2004} Joachim, Michael,
\textit{Toral classes and the Gromov-Lawson-Rosenberg conjecture for elementary abelian 2-groups}, 
Arch. Math. (Basel) 83 (2004), no. 5, 461--466.

\bibitem[JS2000]{JS2000} Joachim, Michael; Schick, Thomas 
\textit{Positive and negative results concerning the Gromov-Lawson-Rosenberg conjecture}, Geometry and topology: Aarhus (1998), 213--226,
Contemp. Math., 258, Amer. Math. Soc., Providence, RI, 2000.

\bibitem[Ju]{Ju} Jung, Rainer, unpublished

\bibitem[KW1975]{KW1975} Kazdan, Jerry L.; Warner, F. W.,  \textit{Existence and conformal deformation of metrics with prescribed Gaussian and scalar curvatures},  Ann. of Math. (2) 101 (1975), 317--331.

\bibitem[KrSt1993]{KrSt1993} Kreck, Matthias and Stolz, Stephan, \textit{$\HP^2$-bundles and elliptic homology}, Acta Math. 171 (1993), 231--261

\bibitem[KwSch1990]{KwSch1990}  Kwasik, Slawomir; Schultz, Reinhard, \textit{Positive scalar curvature and periodic fundamental groups}, Comment. Math. Helv. 65 (1990), no. 2, 271--286

\bibitem[LM1989]{LM1989} Lawson, H. Blaine, Jr.; Michelsohn, Marie-Louise \textit{Spin geometry},  Princeton Mathematical Series, 38. Princeton University Press, Princeton, NJ, 1989. xii+427 pp.

\bibitem[Li1963]{Li1963} Lichnerowicz, André, \textit{Spineurs harmoniques}, C. R. Acad. Sci. Paris 257 (1963), 7--9.

\bibitem[Lo2006]{Lo2006} Lohkamp, Joachim \textit{Positive scalar curvature in dim$\ge 8$}. C. R. Math. Acad. Sci. Paris 343 (2006), no. 9, 585--588.

\bibitem[Miy1985]{Miy1985} Miyazaki, Tetsuro, \textit{Simply connected spin manifolds with positive scalar curvature}. Proc. Amer. Math. Soc. 93 (1985), no. 4, 730–734

\bibitem[Ro1983]{Ro1983} Rosenberg, Jonathan, \textit{$C^*$-algebras, positive scalar curvature, and the Novikov conjecture}, Inst. Hautes \' Etudes Sci. Publ. Math. No. 58 (1983), 197--212 (1984)

\bibitem[Ro1986II]{Ro1986II} Rosenberg, J., \textit{$C^*$-algebras, positive scalar curvature, and the Novikov conjecture. II},  Geometric methods in operator algebras (Kyoto, 1983), 341–374, Pitman Res. Notes Math. Ser., 123, Longman Sci. Tech., Harlow, 1986.

\bibitem[Ro1986III]{Ro1986III}  Rosenberg, Jonathan, \textit{$C^*$-algebras, positive scalar curvature, and the Novikov conjecture. III}, Topology 25 (1986), no. 3, 319–336

\bibitem[Ro1991]{Ro1991} Rosenberg, Jonathan, \textit{The KO-assembly map and positive scalar curvature},  Algebraic topology Poznań 1989, 170--182,
Lecture Notes in Math., 1474, Springer, Berlin, 1991

\bibitem[RS1994]{RS1994}, Rosenberg, Jonathan; Stolz, Stephan, \textit{Manifolds of positive scalar curvature}, Algebraic topology and its applications, 241–267, Math. Sci. Res. Inst. Publ., 27, Springer, New York, 1994

\bibitem[RS1995]{RSt1995} Rosenberg, Jonathan; Stolz, Stephan, \textit{A "stable'' version of the Gromov­ Lawson conjecture}, The Čech centennial (Boston, MA, 1993), 405–418, Contemp. Math., 181, Amer. Math. Soc., Providence, RI, 1995.

\bibitem[RS2001]{RS2001} Rosenberg, Jonathan; Stolz, Stephan, \textit{Metrics of positive scalar curvature and connections with surgery}, Surveys on surgery theory, Vol. 2, 353–386, Ann. of Math. Stud., 149,
Princeton Univ. Press, Princeton, NJ, 2001

\bibitem[Sch1998]{Sch1998}  Schick, Thomas,  \textit{A counterexample to the (unstable) Gromov-Lawson-Rosenberg conjecture}, Topology 37 (1998), no. 6, 1165--1168

\bibitem[St1992]{St1992} Stolz, Stephan, \textit{Simply connected manifolds of positive scalar curvature}, Ann. of Math. (2) 136 (1992), no. 3, 511--540.

\bibitem[St1994]{St1994}  Stolz, Stephan, \textit{Splitting certain MSpin-module spectra},  Topology 33 (1994), no. 1, 159–180.

\bibitem[St1995]{St1995} Stolz, Stephan, \textit{Positive scalar curvature metrics --existence and classification questions}, Proceedings of the International Congress of Mathematicians, Vol. 1, 2 (Zürich, 1994), 625–636, Birkhäuser, Basel, 1995.

\bibitem[St1996]{St1996} Stolz, Stephan, \textit{Concordance classes of \pscm s}, preprint 1996, unpublished, available at \url{https://www3.nd.edu/~stolz/preprint.html}

\bibitem[St2001]{St2001} Stolz, Stephan,  \textit{Manifolds of positive scalar curvature}, Topology of high­ dimensional manifolds, No. 1, 2 (Trieste, 2001), 661--709, ICTP Lect. Notes, 9, Abdus Salam Int. Cent. Theoret. Phys., Trieste, 2002.

\bibitem[SY1979a]{SY1979a}  Schoen, R.; Yau, S. T.,  \textit{On the structure of manifolds with positive scalar curvature}, Manuscripta Math. 28 (1979), no. 1-3, 159--183

\bibitem[SY1979b]{SY1979b} Schoen, R.; Yau, Shing Tung, \textit{Existence of incompressible minimal surfaces and the topology of three-dimensional manifolds with nonnegative scalar curvature}, Ann. of Math. (2) 110 (1979), no. 1, 127--142.

\bibitem[SY2017]{SY2017} Schoen, R.; Yau, Shing Tung, \textit{Positive scalar curvature and minimal hypersurface singularities}, arXiv:1704.0490v1

\bibitem[Wa1960]{Wa1960} Wall, C. T. C.
\textit{Determination of the cobordism ring},
Ann. of Math. (2) 72 (1960), 292–311

\end{thebibliography}
\end{document}